\documentclass[a4paper,12pt]{article}
\usepackage[french, english]{babel}
\usepackage[utf8]{inputenc}
\usepackage[T1]{fontenc}
\usepackage{lmodern}
\usepackage[backend=biber,doi=false,url=false,isbn=false,maxbibnames=99, 
giveninits=true]{biblatex}
\usepackage{csquotes}
\usepackage{geometry}
\usepackage{amsmath,amssymb,amsthm}
\usepackage{tikz}
\usetikzlibrary{calc, positioning, intersections, arrows.meta}
\usepackage[titletoc,title]{appendix}
\usepackage{dsfont}
\usepackage{microtype}
\usepackage{hyperref}

\definecolor{sienne}{RGB}{136, 45, 23}
\hypersetup{
colorlinks = true,
linkcolor = sienne
}

\newtheorem{theorem}{Theorem}
\newtheorem{prop}[theorem]{Proposition}

\newtheorem{lemma}[theorem]{Lemma}

\theoremstyle{definition}
\newtheorem{definition}[theorem]{Definition}

\theoremstyle{remark}
\newtheorem{remark}[theorem]{Remark}
\newtheorem{ex}[theorem]{Example}

\DeclareMathOperator{\distance}{distance}

\newcommand{\diff}[1][-3]{\ensuremath{\mathop{}\mkern#1mu{d}}}
\newcommand{\set}{\mathbb}
\newcommand{\cross}[1]{+(-#1,0) -- +(#1,0) +(0,-#1) -- +(0,#1)}
\allowdisplaybreaks[1]
\title{Non null controllability of the Grushin operator in 2D}
\author{Armand 
KOENIG\thanks{\textit{email:} \texttt{armand.koenig@unice.fr} \textit{address:}
Laboratoire de Mathématiques J.A. Dieudonné
UMR \no 7351 CNRS UNS
Université de Nice - Sophia Antipolis
06108 Nice Cedex 02
France}\\\small Université Côte d'Azur, CNRS, LJAD, France}

\newcommand\blfootnote[1]{%
  \begingroup
  \renewcommand\thefootnote{}\footnote{#1}%
  \addtocounter{footnote}{-1}%
  \endgroup
}

\begin{document}
\maketitle
\begin{abstract}
We are interested in the exact null controllability of the equation 
$\partial_t f - \partial_x^2 f - x^2 \partial_y^2f = \mathds 1_\omega u$, with 
control $u$ supported on $\omega$. We show that, when $\omega$ does not 
intersect a horizontal band, the considered equation is never 
null-controllable. The main idea is to interpret the associated observability 
inequality as an $L^2$ estimate on polynomials, which Runge's theorem 
disproves. To that end, we study in particular the first eigenvalue of the 
operator $-\partial_x^2 + (nx)^2$ with Dirichlet conditions on $(-1,1)$ and we show a quite precise 
estimation it satisfies, even when $n$ is in some complex domain.
\end{abstract}

\blfootnote{
\copyright 2017 This manuscript version is made 
available under the \hbox{\textsc{CC-BY-NC-ND 4.0}} license 
\url{http://creativecommons.org/licenses/by-nc-nd/4.0/}. \textit{Published paper} 
\textsc{doi}: \href{https://doi.org/10.1016/j.crma.2017.10.021}{\tt 
10.1016/\allowbreak  j.crma.2017.10.021}.
This work was partially supported by the ERC advanced grant 
SCAPDE, seventh framework program, agreement \no~320845.}

\setcounter{tocdepth}{1}
\tableofcontents
\section{Introduction}
\subsection{The problem of controllability of the Grushin equation}
We are interested in the following equation, where $\set T = \set R\slash 
2\pi\set Z$, $\Omega = (-1,1)\times \set T$ and $\omega$ is an open subset of 
$\Omega$:
\begin{equation*}
\begin{aligned}
(\partial_t - \partial_x^2 - x^2\partial_y^2)f(t,x,y)  &= \mathds 1_\omega u(t,x,y)
&&\  t\in [0,T], (x,y)\in \Omega\\
f(t,x,y) &= 0 &&\  t \in [0,T], (x,y)\in \partial\Omega.
\end{aligned}
\end{equation*}
It is a control problem with state $f$ and control $u$ supported on $\omega$. More precisely, we are interested in the exact null controllability of this equation:
\begin{definition}\label{def:control}
We say that the Grushin equation is null-controllable on $\omega$ in time $T>0$ if for all $f_0$ in $L^2(\Omega)$, there exists $u$ in $L^2([0,T]\times \omega)$ such that the solution $f$ of:
\begin{equation}\label{eq:control_problem_eng}
\begin{aligned}
(\partial_t - \partial_x^2 - x^2\partial_y^2)f(t,x,y)  &= \mathds 1_\omega u(t,x,y) 
&&\  t\in [0,T], (x,y)\in \Omega\\
f(t,x,y) &= 0 &&\  t \in [0,T], (x,y)\in \partial\Omega\\
f(0,x,y) &= f_0(x,y)&&\ (x,y)\in \Omega
\end{aligned}
\end{equation}
satisfies $f(T,x,y) = 0$ for all $(x,y)$ in $\Omega$.
\end{definition}

We show in this paper that if $\omega$ does not intersect a horizontal band, then the answer is negative whatever $T$ is:
\begin{theorem}\label{th:main}
  Let $[a,b]$ be a non trivial segment of $\set T$ and $\omega = (-1,1)\times 
(\set T\setminus [a,b])$. Let $T>0$. The Grushin equation is not 
null-controllable in $\omega$ on time $T$.
\end{theorem}

That is to say, there exists some $f_0\in L^2(\Omega)$ that no $u\in 
L^2([0,T]\times \omega)$ can steer to $0$ in time $T$. This can be strengthened 
by saying that even if we ask the initial condition to be very 
regular, it may be impossible to steer it to $0$ in finite time. We will state 
this in a precise way in Proposition \ref{th:main_aug}.

We stated Theorem \ref{th:main} with $\Omega = (-1,1)\times \set T$ as it is 
(very) slightly easier than $\Omega = (-1,1)\times (0,1)$, but the situation is the 
same for both cases, and we briefly explain in Appendix \ref{sec:grushin_rectangle} 
what to do for the later case.

The proof we provide for this theorem is very specific to the potential $x^2$: if 
we replace $x^2$ in equation \eqref{eq:control_problem_eng} by, say, 
$x^2 + \epsilon x^3$, we cannot prove with our method the non-null controllability. 
However, there is only a single, but crucial argument that prevents us from doing 
so. We will discuss this a little further after Theorem \ref{th:asymptotic}.

\subsection{Bibliographical comments}
This equation has previously been studied on $(-1,1)\times(0,1)$, and some 
results already exist for different controllability sets. Controllability holds 
for large time but not in small time if $\omega = (a,b)\times(0,1)$ with 
$0<a<b$, as shown by Beauchard, Cannarsa and Guglielmi 
\cite{beauchard_null_2014-1}, and holds in any time if $\omega = 
(0,a)\times(-1,1)$ with $0<a$, as shown by Beauchard, Miller and Morancey 
\cite{beauchard_2d_2015}.

The controllability of the Grushin equation is part of the larger field of the 
controllability of degenerate parabolic partial differential equations of 
hypoelliptic type. For the non-degenerate case, controllability is known since 
1995 to hold when $\Omega$ is any bounded $C^2$ domain, in any open control 
domain and in arbitrarily small time \cite{lebeau_controexact_1995, 
fursikov_controllability_1996}. For parabolic equations degenerating on the 
boundary, the situation is well understood in dimension one 
\cite{cannarsa_carleman_2008} and in dimension two \cite{cannarsa_global_2016}. For 
parabolic equations degenerating inside the domain, in addition to the already 
mentioned two articles on the Grushin equation, we mention articles on 
Kolmogorov-type equations \cite{beauchard_null_2014,beauchard_degenerate_2015}, the 
heat equation on the Heisenberg group \cite{beauchard_heat_2017} and quadratic 
hypoelliptic equations on the whole space \cite{beauchard_null-controllability_2016, 
beauchard_null-controllability_2016-1}.

\subsection{Outline of the proof, structure of the article}\label{sec:outline}
As usual in controllability problems, we focus on the following observability 
inequality on the adjoint system, which is equivalent by a duality argument to the 
null-controllability (definition \ref{def:control}, see \cite[][Theorem 
2.44]{coron_control_2007} for a proof of this equivalence): there exists $C>0$ 
such that for all $f_0$ in $L^2(\Omega)$, the solution $f$ of:
\begin{equation}\label{eq:grushin}
\begin{aligned}
(\partial_t - \partial_x^2 - x^2\partial_y^2)f(t,x,y)  &= 0 
&&\  t\in [0,T], (x,y)\in \Omega\\
f(t,x,y) &= 0 &&\  t \in [0,T], (x,y)\in \partial\Omega\\
f(0,x,y) &= f_0(x,y) &&(x,y)\in \Omega
\end{aligned}
\end{equation}
satisfies:
\begin{equation}\label{eq:observability}
 \int_\Omega |f(T,x,y)|^2\diff x\diff y  \leq C \int_{[0,T]\times \omega} 
|f(t,x,y)|^2\diff t \diff x \diff y.
\end{equation}

Therefore, the Theorem \ref{th:main} can be stated the following way :
\begin{theorem}\label{th:main_obs}
 There exists a sequence $(f_{k,0})$ in $(L^2(\Omega))^{\set N}$ such that for 
every $k\in \set N$, the solution $f_k$ of the Grushin equation 
\eqref{eq:grushin} with initial condition $f_{k,0}$ satisfies
$\sup_k|f_k|_{L^2([0,T]\times \omega)} <+\infty$ and $|f_k(T,\cdot, 
\cdot)|_{L^2(\Omega)}\to +\infty$ as $k\to +\infty$. 
\end{theorem}

To find such a sequence, we look for solutions of the Grushin 
equation~\eqref{eq:grushin} that concentrate near $x=0$. To that end, we remark that, 
denoting $v_{n,k}$ an eigenfunction of the operator $-\partial_x^2 + (n x)^2$ with 
Dirichlet boundary conditions on $(-1,1)$ associated with eigenvalue $\lambda_{n,k}$, 
$\Phi_{n,k}(x,y) = v_{n,k}(x) e^{iny}$ is an eigenfunction of the Grushin operator 
$-\partial_x^2 -x^2\partial_y^2$ with eigenvalue $\lambda_{n,k}$. In addition, we 
expect that the first eigenfunction $v_n = v_{n,0}$ of $-\partial_x^2 +(nx)^2$ on 
$(-1,1)$ is close to the first eigenfunction of the same operator on $\set R$, that 
is $v_n \sim (\frac n{4\pi})^{1/4} e^{-nx^2/2}$, and that the associated eigenvalue 
$\lambda_n=\lambda_{n,0}$ is close to $n$.
So it is natural to look for a counterexample of the observability inequality 
\eqref{eq:observability} among the linear combinations of $\Phi_n(x,y) = v_n(x) 
e^{iny}$ for $n\geq 0$.

In Section \ref{sec:toy_model}, we will see by heuristic arguments and with the help 
of these approximations that the problem 
of the controllability of the Grushin equation is close to 
the controllability of the square root of minus the Laplacian, and show that this 
model is not null controllable. As another warm-up, we will show in section 
\ref{sec:toy_to_real} that the method used for treating the square root of minus the 
Laplacian allows us to treat with little changes the case of the Grushin equation for 
$(x,y)\in \set R\times \set T$.

The case of the Grushin equation for $(x,y)\in 
(-1,1)\times \set T$ (Theorem \ref{th:main}) gave us much more trouble, but in 
Section \ref{sec:main_proof} we are able to adapt the method used in the previous 
two cases. To achieve that, we use some technical tools that are proved in later 
sections. First, in Section \ref{sec:est_def_op}, estimates on polynomials of the 
form $|\sum \gamma_n a_n z^n|_{L^\infty(U)} \leq C |\sum a_n 
z^n|_{L^\infty(U^\delta)}$, under a simple geometric hypothesis on $U$, and some 
general---although somewhat hard to prove---hypotheses 
on the sequence $(\gamma_n)$ (Theorem \ref{th:est_default}). Second, in Section 
\ref{sec:spectral}, a spectral analysis of the operator $-\partial_x^2 + (n 
x)^2$ on $(-1,1)$; most importantly an asymptotic expansion of the first eigenvalue 
$\lambda_n$ of the form $\lambda_n = n + \gamma(n) e^{-n}$ with $\gamma(n) \sim 4 
\pi^{-1/2}n^{3/2}$, and $\gamma$ having a particular holomorphic structure (Theorem 
\ref{th:asymptotic}).

\section{Proof of the non null controllability of the Grushin equation}
\subsection{The toy model}\label{sec:toy_model}

Let us write the observability inequality on functions of the form $\sum a_n v_n(x) 
e^{iny}$ (keeping in mind that $v_n(x)$ is real, and noting $\omega_y = \set 
T\setminus [a,b]$ so that $\omega = (-1,1)\times 
\omega_y$):
\begin{equation}
    \sum_{n} |a_n|^2 e^{-2\lambda_nT} \leq
    C\sum_{n,m} a_n\overline{a_m} \int_{-1}^1 v_n(x) v_m(x)\diff x
    \int_{0}^T e^{-(\lambda_n + \lambda_m)t}\diff t
    \int_{y\in \omega_y} \mkern-15mue^{i(n-m)y}\diff y.
\end{equation}

Now let us proceed by heuristic arguments to see what we can expect from the  
estimates on the eigenvalues $\lambda_n$ and the eigenfunctions $v_n$ we mentioned 
in Section \ref{sec:outline}. We imagine 
that in the previous inequality, $\lambda_n = n$ and $\int_{-1}^1 v_nv_m =  
\frac1{\sqrt{4\pi}}(nm)^{1/4} \int_{\set R}e^{-(n+m)x^2/2}\diff x = \sqrt 
2\frac{(nm)^{1/4}}{\sqrt{n+m}}$, which does not decay very fast off-diagonal, so we 
further imagine that $\int_{-1}^1 v_n v_m = 1$. Then, with these approximations, the 
previous observability inequality reads:
\begin{equation}\label{eq:model_observability}
\begin{split}
  \sum_n |a_n|^2 e^{-2nT} \leq C \sum_{n,m} a_n \overline{a_m}
  \int_{[0,T]\times \omega_y} e^{-(n+m)t +i(n-m)y} \diff t \diff y \\
  =C\int_{[0,T]\times \omega_y} \left|\sum a_ne^{-nt+iny}\right|^2\diff t \diff y.
  \end{split}
\end{equation}

This suggests that the controllability problem of the Grushin equation 
\eqref{eq:control_problem_eng} is similar to the following model control problem: 
let us consider the Hilbert space $\{\sum_{n\geq 0} a_n e^{iny}, \sum |a_n|^2 
<+\infty\}$, $D$ the unbounded operator on this space with domain $\{\sum 
a_n e^{iny},\sum n^2|a_n|^2 <+\infty\}$ defined by $D(\sum a_ne^{iny}) = 
\sum na_n e^{iny}$.  Then the null controllability of the equation $(\partial_t + 
D)f = \mathds1_\omega u$ on an open set $\omega = \set T\setminus[a,b]$ in time $T$ 
is equivalent to the previous ``simplified'' observability inequality 
\eqref{eq:model_observability}, which does not hold:
\begin{theorem}
  Let $[a,b]$ be a nontrivial segment of $\set T$, $\omega_y = \set T \setminus 
[a,b]$ and $T>0$. The equation $(\partial_t+D)f = \mathds1_{\omega_y} u$ is not 
null controllable on $\omega_y$ in time $T$.
\end{theorem}

Incidentally, this is an answer to a specific case of an open problem mentioned 
by Miller \cite[][section 3.3]{miller_controllability_2006} and again by 
Duyckaerts and Miller \cite[][remark 6.4]{duyckaerts_resolvent_2012}.

\begin{proof}
The right hand side of the observability inequality \eqref{eq:model_observability} 
suggests to make the change of variables $z=e^{-t+iy}$, for which\footnote{We denote 
$\lambda$ the Lebesgue measure on $\set C$; that is if 
$(\mu,\nu)\mapsto f(\mu+i\nu)$ is integrable on $\set R^2$, $\int_{\set C} 
f(z)\diff\lambda(z) = \int_{\set 
R^2} f(\mu + i\nu)\diff\mu \diff\nu $.} $\diff[0] t\diff y = 
|z|^{-2}\diff\lambda(z)$, 
and that maps $[0,T]\times \omega_y$ to $\mathcal D = \{e^{-T}<|z|<1, \arg(z)\in 
\omega_y\}$ (see figure \ref{fig:model}). So, the 
right hand side of the observability inequality~\eqref{eq:model_observability} is 
equal to:
\begin{equation}\label{eq:obs_hol_rhs_R}
\int_{[0,T]\times \omega_y} \left|\sum a_n 
e^{-nt+iny}\right|^2\diff t\diff y = \int_{\mathcal D} \left|\sum a_n 
z^{n}\right|^2|z|^{-2} \diff\lambda(z).
\end{equation}

About the left hand side, we first note that by writing the integral on a disk 
$D=D(0,r)$ of $z^n\bar z^m$ in polar coordinates, we find that the functions 
$z\mapsto z^n$ are orthogonal on $D(0,r)$. So, we have for all 
polynomials $\sum_{n\geq 1} a_n z^n$ with a zero at $0$: 
\[\int_{D(0,e^{-T})} \left|\sum a_n z^n\right|^2 |z|^{-2} \diff \lambda(z) 
= \sum |a_n|^2 \int_{D(0,e^{-T})} |z|^{2n-2} \diff \lambda(z)\]
and, combined with the fact that by another computation in polar coordinates, for 
$n\geq 1$, $\int_{D(0,e^{-T})}|z|^{2n-2}\diff \lambda(z) = \frac\pi ne^{-2nT}$:
\begin{equation}\label{eq:obs_hol_lhs_R}
\int_{D(0,e^{-T})} \left|\sum a_n z^n\right|^2 |z|^{-2} \diff \lambda(z) 
\leq \pi \sum |a_n|^2 e^{-2nT}.
\end{equation}

So, thanks to equations \eqref{eq:obs_hol_rhs_R} and \eqref{eq:obs_hol_lhs_R}, the 
observability inequality \eqref{eq:model_observability} implies that for some 
$C'>0$ and for all polynomials $f$ with $f(0) = 0$:
\[\int_{D(0,e^{-T})} |f(z)|^2|z|^{-2} \diff \lambda(z) \leq C' 
\int_{\mathcal D} |f(z)|^2|z|^{-2} \diff \lambda(z).\]

By the change of indices $n' = n-1$ in the sum $f(z) = \sum_{n\geq 1} a_n z^n$, we 
rewrite this ``holomorphic observability inequality'' in the following, slightly 
simpler way: for every polynomials $f$:
\begin{equation}\label{eq:model_hol_obs}
  \int_{D(0,e^{-T})} |f(z)|^2\diff \lambda(z) \leq C' \int_{\mathcal D} 
|f(z)|^2\diff \lambda(z).
\end{equation}

This is the main idea of the proof: the observability inequality of the control 
problem is almost the same as an $L^2$ estimate on polynomials. We will disprove 
it thanks to Runge's theorem, whose proof can be found in 
Rudin's famous textbook \cite[][theorem 13.9]{rudin_real_1986}. More 
specifically, we will need the following special case:
\begin{prop}[Runge's theorem]\label{th:runges_th}
  Let $U$ be a connected and simply connected open subset of $\set C$, and let $f$ 
be an holomorphic function on $U$. There exists a sequence $(f_n)$ of polynomials 
that converges uniformly on every compact subsets of $U$ to $f$.
\end{prop}

Let $\theta \in \set T$ non-adherent to $\omega_y$ (for instance $\theta = 
(a+b)/2)$). We choose in the previous theorem $U = \set C \setminus e^{i\theta}\set 
R_+$ (see figure \ref{fig:model}) and $f(z) = \frac1z$. Since $z\mapsto\frac1z$ is 
bounded on $\mathcal D$, $f_n$ is uniformly bounded on $\mathcal D$ and the right 
hand side of the holomorphic observability inequality \eqref{eq:model_hol_obs} 
$\int_{\mathcal D} |f_n|^2\diff\lambda(z)$ stays bounded. But since 
$z\mapsto\frac1z$ has infinite $L^2$ norm on $D(0,e^{-T})$, and thanks to Fatou's 
lemma, the left hand side $\int_{D(0,e^{-T})} |f_n|^2\diff\lambda(z)$ tends to 
infinity as $n$ tends to infinity. \qedhere

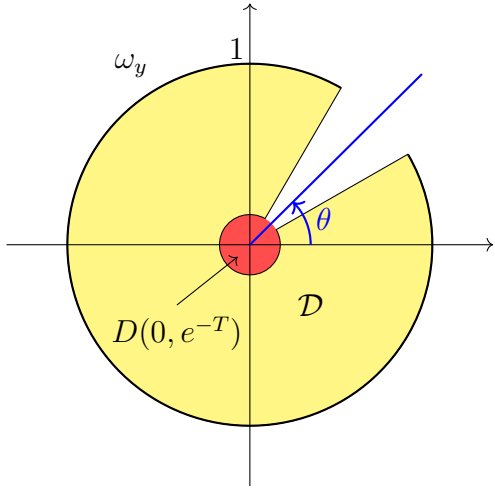
\begin{figure}
\begin{minipage}[c]{0.5\textwidth}
 \begin{center}
 \begin{tikzpicture}[scale=0.8]
  \fill[red!70] (0,0) circle[radius=0.5];
  \draw[fill = yellow!60] (60:0.5)
  arc[radius = 0.5, start angle = 60, delta angle = 330] --
  ($(0,0) +(30:3)$) arc[radius = 3, start angle = 
30, delta angle = -330] -- cycle;
  \draw[thick] (60:3) arc[radius = 3, start angle=60, delta angle=330];
  \node[above left] at (120:3) {$\omega_y$};
  \draw[->] (-4,0) -- (4,0);
  \draw[->] (0,-4) -- (0,4);
  
  \draw[thick, blue] (0,0) -- (45:4);
  \draw[->,thick, blue] (0:1) arc [radius = 1, start angle=0, delta angle=45];
  \node[right, blue] at (25:1) {$\theta$};
  \path (0,3) node [anchor=south east, inner sep = 0.15 em]{$1$};
  \path (1,-1) node{$\mathcal D$};
  \draw[->] (-1.2,-1) node[below]{$D(0,e^{-T})$} -- (-0.2,-0.2);
\end{tikzpicture}
\end{center}
\end{minipage}\hfill%
\begin{minipage}[c]{0.5\textwidth}
\caption{In yellow, the domain $\mathcal D$, in red, the disk $D(0,e^{-T})$. The 
thick outer circular arc is the subset $\omega_y$ of $\set T \simeq \set U$. The 
controllability of the model operator $D$ on $\omega_y$ in time $T$ would 
imply the the control of the $L^2(D(0,e^{-T}))$-norm of polynomials by their 
$L^2$ norm on $\mathcal D$.}\label{fig:model}
\end{minipage}
\end{figure}
\end{proof}

\begin{remark}
 This proof is specific to the one dimensional case, as it relies on the observation 
that the 
solutions of the equation $(\partial_t + D)f = 0$ are holomorphic in $z=e^{-t+iy}$. 
As far as the author knows, this argument does not generalize to higher dimension.
\end{remark}

\subsection{From the toy model to the Grushin equation: the 
case of the Grushin equation on $\set R\times \set T$}\label{sec:toy_to_real}
We show here that the method we used for the toy model is also effective to prove 
that the Grushin equation for $(x,y)\in \set R\times\set T$, i.e. the equation
\begin{equation}\label{eq:grushin_R}
\begin{aligned}
(\partial_t - \partial_x^2 - x^2\partial_y^2)f(t,x,y)  &= 0 
&&\  t\in [0,T], (x,y)\in \set R\times \set T\\
f(t,\cdot,\cdot) &\in L^2(\set R\times \set T)\\
f(0,x,y) &= f_0(x,y) &&(x,y)\in \Omega,
\end{aligned}
\end{equation}\nopagebreak%
where we choose $\omega = \set R \times \omega_y = \set R\times (\set T\setminus 
[a,b])$, is not null-controllable in any time.

In this case, the (unbounded) operator $-\partial_x^2 +(nx)^2$ on $L^2$ is perfectly 
known: 
its first eigenvalue is $n$ and the associated eigenfunction\footnote{We choose to 
normalize the eigenfunction so that $v_n(0) = 1$ instead of $|v_n|_{L^2(\set R)} = 
1$, that way, the proof will be slightly easier. Note that with this choice of 
normalization, we have $|v_n|_{L^2(\set R)} = (\pi/n)^{1/4}$.} is $v_n(x) = e^{-n 
x^2/2}$. So 
the functions $\Phi_n$ defined by $\Phi_n(x,y) = e^{-n x^2/2} e^{iny}$ are 
eigenfunctions of the operator $\partial_x^2 +x^2 \partial_y^2$ with 
respective eigenvalue $n$.

Let us write the associated observability inequality on the solutions of the Grushin 
equation of the form $\sum a_n e^{-nt}\Phi_n(x,y)$, where the sum has 
finite support\footnote{We could extend all the following estimates by density to 
some functional spaces, but we won't need to, as we didn't need to extend the 
estimate \eqref{eq:model_hol_obs} to other functions than polynomials.}:

\begin{equation}\label{eq:obs_grushin_R}
  \sum \left(\frac\pi n\right)^{1/2}|a_n|^2 e^{-2n T}
  \leq C \int_{x\in\set R} \left(\int_{\substack{t \in [0,T]\\y\in \omega_y}} 
\left|\sum 
 a_n e^{n\left(-\frac{x^2}2 -t + iy\right)}\right|^2\diff t \diff y\right) \diff x.
\end{equation}

The first difference between this observability inequality, that we 
try to 
disprove, and the observability inequality of the toy model 
\eqref{eq:model_hol_obs}, is the factor $(\pi/n)^{1/2}$, but it is not a real 
problem. The 
main difference is 
the presence of another variable: $x$. For each $x$, the term $e^{-nx^2/2}$ 
acts as a contraction of $\mathcal D$, so we make the change of variable that takes 
into account this contraction $z_x = e^{-x^2/2 -t + iy}$. We have $\diff[-6] t\diff 
y = |z_x|^{-2}\diff \lambda(z_x)$, and this change of variables sends 
$(0,T)\times \omega_y$ to $e^{-x^2/2}\mathcal D$, with, as in the toy 
model, $\mathcal D = \{e^{-T} < |z| < 1, \arg(z)\in \omega_y\}$:
\begin{equation*}
 \sum \left(\frac\pi n\right)^{1/2}|a_n|^2 e^{-2nT} \leq C\int_{x\in \set R} 
\int_{e^{-x^2/2}\mathcal D} 
\left|\sum a_n z^{n-1}\right|^2 \diff \lambda(z)\diff x.
\end{equation*}
We have seen in the toy model that for all polynomials $f(z)=\sum_{n\geq 1} a_n z^n$ 
with $f(0) = 0$ that $\int_{D(0,e^{-T})} |f(z)|^2|z|^{-2}\diff \lambda(z) = 
\pi\sum\frac1n|a_n|^2e^{-2nT}$, which is smaller than the left hand side of the 
observability inequality \eqref{eq:obs_grushin_R}, up to a constant $\sqrt \pi$. 
So, as 
in the toy model, this would imply that for all polynomials $f$:
\begin{equation}\label{eq:obs_hol_grushin_R}
 \int_{D(0,e^{-T})} |f(z)|^2\diff\lambda(z) \leq \sqrt \pi C\int_{x\in \set R} 
\int_{e^{-x^2/2}\mathcal D} |f(z)|^2\diff \lambda(z)\diff x.
\end{equation}

We want to apply the same method as the one used in the toy model to disprove 
this inequality, but we have to be a little careful: the right hand side exhibits an 
integrals over $e^{-x^2/2}\mathcal D$, and as $x$ tends to infinity, $0$ becomes 
arbitrarily close to the integration set. So, instead of choosing a sequence of 
polynomials that blows up at $z=0$, we choose one that blows up away from $0$ and 
from every $e^{-x^2/2}\mathcal D$. More precisely, we choose $\theta\notin 
\overline{\omega_y}$, $z_0 = e^{i\theta-2T}$, and $f_k$ a sequence of polynomials 
that converges to $z\mapsto (z-z_0)^{-1}$ uniformly on every compact of $\set 
C\setminus (z_0[1,+\infty[)$ (see Figure \ref{fig:grushin_R}).
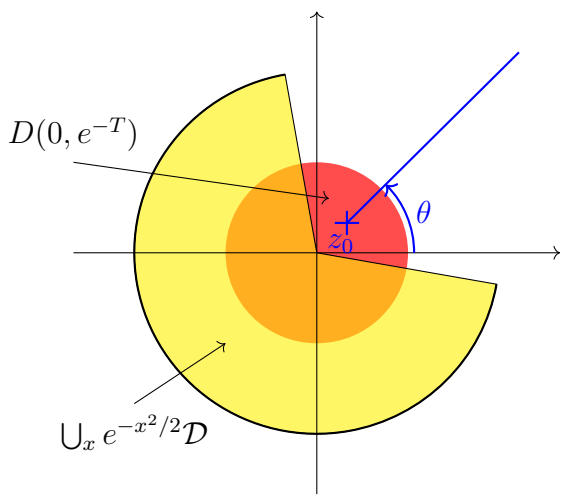
\begin{figure}[ht]
\begin{minipage}[c]{0.57\textwidth}
 \begin{center}
 \begin{tikzpicture}[scale=0.8]
  \fill[red, fill opacity = 0.7] (0,0) circle[radius=1.5];
  \draw[fill = yellow, fill opacity = 0.6] (0,0) --
  ($(0,0) +(-10:3)$) arc[radius = 3, start angle = 
-10, delta angle = -250] -- cycle;
  \draw[thick] (100:3) arc[radius = 3, start angle=100, delta angle=250];
  \draw[->] (-4,0) -- (4,0);
  \draw[->] (0,-4) -- (0,4);
  
  \draw[->,thick, blue] (0:1.6) arc [radius = 1.6, start angle=0, delta angle=45];
  \node[blue,right] at (25:1.6) {$\theta$};
  \draw[blue, thick] (45:0.7) +(-0.2,0) -- +(0.2,0) +(0,0.2) -- +(0,-0.2) +(0,0) -- 
+(45:4) +(-0.1,0) node[below]{$z_0$};
  \draw[->] (-3,-2.5) node[below]{$\bigcup_x e^{-x^2/2}\mathcal D$} -- (-1.5,-1.5);
  \draw[->] (-4,1.5) node[above]{$D(0,e^{-T})$} -- (0.2,0.9);
\end{tikzpicture}
\end{center}
\end{minipage}\hfill%
\begin{minipage}[c]{0.43\textwidth}
\caption{The equivalent of Figure \ref{fig:model} for the Grushin 
equation. Again in red, the disk $D(0,e^{-T})$, and in yellow the union of 
$e^{-x^2/2}\mathcal D$, which is the ``pacman'' $\{0<|z|<1, \arg(z)\in \omega_y\}$. 
We choose a sequence of polynomials that converges to $z\mapsto 
(z-z_0)^{-1}$ away from the blue half-line.}\label{fig:grushin_R}
\end{minipage}
\end{figure}

With the same argument as in the toy model, we know that the left-hand side
$\int_{D(0,e^{-T})}|f_k(z)|^2\diff\lambda(z)$ tends to infinity as $k$ tends to 
infinity. As for the right hand-side, since $z\mapsto (z-z_0)^{-1}$ is bounded in 
$\bigcup_x e^{-x^2/2}\mathcal D = \{0<|z|<1, \arg(z)\in\omega_y\}$, $f_k$ is bounded 
on $e^{-x^2/2}\mathcal D$ uniformly in $x\in \set R$ and $k\in \set N$ by some $M$. 
So, the right hand side satisfies\footnote{Let us remind that $\lambda$ is the 
Lebesgue measure on $\set C$, so, for $A\subset \set C$ measurable, $\lambda(A)$ is 
the area 
of~$A$.}:

\begin{align*}
 \int_{x\in \set R} \int_{e^{-x^2/2}\mathcal D} |f(z)|^2\diff \lambda(z)\diff x
 &\leq \int_{x\in \set R} \int_{e^{-x^2/2}\mathcal D} M^2 \diff\lambda(z) \diff x \\
 &\leq \int_{x\in \set R} \lambda(e^{-x^2/2}\mathcal D)M^2 \diff x\\
 &\leq \int_{x\in \set R} \pi e^{-x^2} M^2 \diff x\\
 &\leq \pi^{3/2} M^2.
\end{align*}

We have proved that the left hand side of inequality \eqref{eq:obs_hol_grushin_R} 
applied to $f = f_k$ tends to infinity as $k$ tends to infinity while it's right 
hand side stays bounded, thus, this inequality is false, and the Grushin equation 
for $(x,y)\in \set R\times \set T$ is never null-controllable in $\omega = \set 
R\times(\set T\setminus [a,b])$.\qed

\subsection{The case of the Grushin equation on $(-1,1)\times \set 
T$}\label{sec:main_proof}
Here we show the main theorem. In comparison with the previous case, we have two 
difficulties: $\lambda_n$ is not exactly $n$, and $v_n(x)$ is not exactly 
$e^{-nx^2/2}$. Let us write the observability 
inequality\footnote{As in 
the previous cases, all the sums are supposed with finite support. We 
could extend by density all the inequalities that follow, but we won't need to.} on 
$\sum a_n e^{-\lambda_n t}v_n(x)e^{iny}$, where $\lambda_n = n+\rho_n$:
\begin{equation}\label{eq:obs}
  \sum |v_n|^2_{L^2(-1,1)}|a_n|^2 e^{-2\lambda_nT}
  \leq C \int_{\substack{t\in[0,T]\\ x\in(-1,1)\\y\in \omega_y}} 
  \left|\sum a_n v_n(x) e^{n(-t+iy)}e^{-\rho_nt}\right|^2\diff t\diff y\diff x.
\end{equation}

As in the previous two cases, the first step is to relate this inequality to an 
estimate on polynomials:
\begin{prop}\label{th:obs_hol}
Let $U = \{0<|z|<1, \arg(z)\in \omega_y\}$, let $\delta>0$ and $U^\delta = \{{z\in 
\set C},\allowbreak \distance(z,U) < \delta\}$ (see figure \ref{fig:mathcal_D}).

 The observability inequality of the Grushin equation implies that there exists
$C'>0$ and an integer $N$ such that for all polynomials $f(z)=\sum_{n > N} a_n 
z^n$ with at least the $N$ first derivatives vanishing at zero\footnote{This 
condition is not really needed, but it makes some theorems less cumbersome to 
state.}:
\begin{equation}\label{eq:obs_hol}
 |f|_{L^2(D(0,e^{-T}))} \leq C'|f|_{L^\infty(U^\delta)}.
\end{equation}
\end{prop}

\begin{proof}
About the left hand side of the observability inequality \eqref{eq:obs}, we remark 
that it is almost the same as in the toy model. Indeed, if $a_0 = 0$, we have seen in 
the proof of the non null controllability of the toy model that 
$\int_{D(0,e^{-T})} |\sum a_n z^{n-1}|^2\diff \lambda(z) = \sum 
\frac\pi n |a_n|^2 e^{-2nT}$. And since $|v_n|_{L^2(-1,1)}^2$ is greater than 
$cn^{-1/2}$ for some $c>0$ (see Lemma \ref{th:lemma2} for a proof), we have: 
\begin{equation*}
\int_{D(0,e^{-T})} \left |\sum a_n z^{n-1}\right|^2 \diff \lambda(z)\leq \pi c^{-1} 
\sum |v_n|_{L^2(-1,1)}^2 |a_n|^2 e^{-2nT}.
\end{equation*}

Moreover, reminding that $\lambda_n = n+\rho_n$, we know that $(\rho_n)$ is 
bounded (see Theorem \ref{th:asymptotic} or 
\cite[][section 3.3]{beauchard_null_2014-1} for a simpler proof). 
So, 
$e^{-2nT} \leq e^{2\sup_k(\rho_k)T}e^{-2\lambda_nT}$. We use that to bound the right 
hand-side of the previous inequality:
\begin{equation}\label{eq:obs_hol_lhs}
\int_{D(0,e^{-T})} \left|\sum a_n z^{n-1}\right|^2 \diff \lambda(z) \leq 
\pi c^{-1} 
e^{2\sup_k(\rho_k) T}\sum|v_n|^2_{L^2(-1,1)}|a_n|^2 e^{-2\lambda_n T}.
\end{equation}

We now want to bound from 
above the right hand side of the observability inequality \eqref{eq:obs} by 
$C'\left|\sum a_n z^{n-1}\right|_{L^\infty(U^\delta)}^2$ for some $C'$. We 
make the change of variables $z = e^{-t+iy}$:
\begin{multline}\label{eq:obs_hol_rhs}
 \int_{\substack{t\in[0,T]\\ x\in (-1,1)\\y\in \omega_y}} 
  \left|\sum a_n v_n(x) e^{n(-t+iy)}e^{-\rho_nt}\right|^2\diff t\diff y\diff x 
=\\[-0.5em]
 \int_{x\in(-1,1)}\left( \int_{z\in \mathcal D} 
  \left|\sum a_n v_n(x) z^{n-1}|z|^{\rho_n}\right|^2\diff\lambda(z) 
  \right)\diff x.
\end{multline}

As in the case of the Grushin equation over $\set R\times \set 
T$ studied in the 
previous section, there is a multiplication by $v_n(x)$. But this time, the action 
of this multiplication
is a little more complicated than just a contraction by a factor $e^{-x^2/2}$. The 
other difficulty is the factor $e^{-\rho_n t} = |z|^{\rho_n}$, which does not seem 
to be a big issue at a first glance, as it is close to $1$; but since it is not 
holomorphic, it is 
actually the biggest issue we are facing. To be able adapt the method used 
in the previous cases, we need to somehow estimate the sum $|\sum 
v_n(x) |z|^{\rho_n} a_n z^n|_{L^2(\mathcal D)}$ by an appropriate norm of $\sum a_n 
z^n$. The Theorem \ref{th:est_default} hinted in the outline gives us such an 
estimate, with the spectral analysis of Section \ref{sec:spectral} giving us the 
required hypotheses. More precisely, we prove in Section \ref{sec:proof_lemma} the 
following lemma:
\begin{lemma}\label{th:lemma1}
There exists an integer $N$ and $C_2>0$ such that for every $x\in (-1,1)$, $z$ and 
$\zeta$ in $\mathcal D$, and every polynomial $\sum_{n>N} a_n z^n$ with derivatives 
up to order $N$ vanishing at $0$:
\[\left|\sum v_n(x) a_n z^{n-1} |\zeta|^{\rho_n}\right|\leq C_2 \left|\sum a_n 
z^{n-1}\right|_{L^\infty(U^\delta)}.\]
\end{lemma}

Applying the above lemma for $z= \zeta$, and assuming that $a_n = 0$ when 
$n\leq N$, we have for every $z\in \mathcal D$: \[\left|\sum v_n(x) a_n 
z^{n-1}|z|^{\rho_n}\right|\leq C_2 \left|\sum a_n 
z^{n-1}\right|_{L^\infty(U^\delta)}\]
so, the right hand side of the observability 
inequality satisfies:
\begin{align*}
&\,\quad\int_{x\in(-1,1)}\left( \int_{\substack{t\in[0,T]\\y\in\omega_y}} 
  \left|\sum a_n v_n(x) e^{n(-t+iy)}e^{-\rho_nt}\right|^2\diff t\diff y
  \right)\diff x \\*
 &=\int_{x\in(-1,1)}\left( \int_{z\in \mathcal D} 
  \left|\sum a_n v_n(x) z^{n-1}|z|^{\rho_n}\right|^2\diff\lambda(z) 
  \right)\diff x &\mkern-50mu\text{(equation \eqref{eq:obs_hol_rhs})}\\
 &\leq \int_{x\in(-1,1)}\left( \int_{z\in \mathcal D} 
  C^2_2\left|\sum 
a_{n+1}z^n\right|_{L^\infty(U^\delta)}^2\diff\lambda(z)\right)\diff x
  &\mkern-50mu\text{(previous lemma)}\\
  &\leq 2C_2^2 \pi\left|\sum a_{n+1}z^n\right|_{L^\infty(U^\delta)}^2
  &\mkern-50mu(\text{area}(\mathcal D) \leq \pi).
\end{align*}

So, together with equation \eqref{eq:obs_hol_lhs} on the left hand side of the 
observability inequality, we have proved that the observability inequality implies 
that for all polynomials $f = \sum_{n\geq N} a_n z^{n-1}$:
\[|f|_{L^2(D(0,e^{-T}))}^2 \leq 2\pi^2e^{2\sup_k(\rho_k)T} c^{-1}CC_2^2 
|f|_{L^\infty(U^\delta)}^2.\qedhere\]
\end{proof}

\begin{figure}
\begin{minipage}[c]{0.6\textwidth}
\begin{center}
   \begin{tikzpicture}[scale=0.8]
  \fill[red, fill opacity = 0.7] (0,0) circle[radius=1.4];
  \draw[fill = yellow, fill opacity = 0.6] (45:0.3) -- 
   ($(-10:3.5)+(80:0.3)$)
   arc[start angle = 80, end angle = -10, radius = 0.3]
   arc[start angle = -10, delta angle = -250, radius = 3.8]
   arc[start angle = 100, end angle = 10, radius = 0.3] --
   cycle;
  \draw[thick] (100:3.5) arc[radius = 3.5, start angle=100, delta angle=250];
  \draw[->] (130:4.3) node[above left]{$\omega_y$} -- (130:3.5);
  \draw[->] (-4,0) -- (4,0);
  \draw[->] (0,-4) -- (0,4);
  
  \draw[blue,thick] (45:0.7) -- (45:4.5) node[pos = 0.8,right]{$z_0[1,+\infty)$};
  \draw[blue] (45:0.7) \cross{0.2} +(0,0.1)node[above]{$z_0$};
  \draw[->,thick, blue] (0:1.2) arc [radius = 1.2, start angle=0, delta angle=45];
  \node[right, blue] at (25:1.2) {$\theta$};
  \path (1.6,-1.6) node{$U^\delta$};
  \draw[->] (-1.3,-1.2) node[below]{$D(0,e^{-T})$} -- (-0.4,-0.4);
  \draw[{Stealth[length = 3mm, reversed]}-{Stealth[length = 3mm, reversed]}] 
($(160:3.5)- (160:3mm)$) -- ($(160:3.8) + (160:3mm)$);
  \draw (160:3.8) -- (160:4.5) node[above right = -0.6ex and 0.1ex]{$\delta$};
\end{tikzpicture}
\end{center}
\end{minipage}\hfill%
\begin{minipage}[c]{0.4\textwidth}
\caption{In yellow, the domain $U^\delta$, in red, the disk 
$D(0,e^{-T})$ and in blue, the point $z_0$ and the half-line 
$z_0[1,+\infty)$. Since $f_k$ converges to $z\mapsto 
z^{N+1}(z-z_0)^{-1}$ away from the blue line, the $L^\infty$ norm of $f_k$ over 
$U^\delta$ is bounded, as long as $\delta<\distance(z_0,\mathcal 
D)$.}\label{fig:mathcal_D}
\end{minipage}
\end{figure}
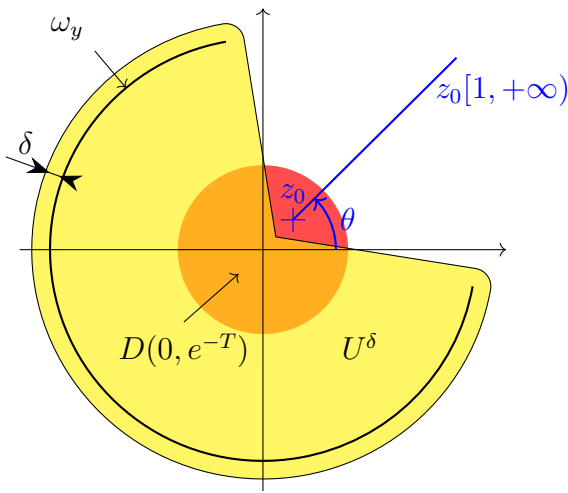

We can find an counterexample of the inequality of the previous proposition exactly 
in the same way we disproved the null controllability of the Grushin equation over 
$\set R\times \set T$:
\begin{proof}[Proof of Theorem \ref{th:main}]
  First we choose $0<\delta<e^{-T}$, so that $D(0,e^{-T})\not \subset U^\delta$.
We also choose $\theta$ non adherent to $\omega_y$, and $z_0 = re^{i\theta}$ with 
$r\in 
(\delta, e^{-T})$ (so that $z_0 \in D(0,e^{-T})$ but $z_0\notin 
\overline{U^\delta}$, see figure \ref{fig:mathcal_D}). Then we choose $\tilde f_k$ a 
sequence of polynomials that converges uniformly on every compact subset of $\set C 
\setminus z_0[1,+\infty)$ to $z\mapsto (z-z_0)^{-1}$. Finally, to satisfy the 
condition of ``enough vanishing derivatives at $0$'' of the previous proposition, we 
chose $f_k(z) = z^{N+1} \tilde f_k(z)$ with $N$ given by the previous proposition. 
This sequence tends to $z\mapsto z^{N+1}(z-z_0)^{-1}$.

Then, again by Fatou's lemma, $|f_k|_{L^2(D(0,e^{-T}))} \to +\infty$ as $k\to 
+\infty$, and since $z\mapsto z^{N+1}(z-z_0)^{-1}$ is bounded on $U^\delta$, $f_k$ 
is uniformly bounded on $U^\delta$. Therefore, the inequality 
$|f_k|_{L^2(D(0,e^{-T}))} \leq C |f_k|_{L^\infty(U^\delta)}$ is false for $k$ 
large enough, and 
according to the previous theorem, so is the observability inequality.
\end{proof}

\section{Estimates for the holomorphy default operators}\label{sec:est_def_op}
\subsection{Symbols}
In this section and the following, we study some operators on 
polynomials of the form $\sum a_n z^n \mapsto \sum \gamma_n a_n z^n$. Since these 
operators make the link between the holomorphy of the solution of the toy model (in 
the variable $z = e^{-t+iy}$) and the solutions of the real Grushin equation (see 
Lemma \ref{th:lemma1}), we will call them \emph{holomorphy default operators}. We 
will also call the sequence $(\gamma_n)$ the \emph{symbol} of the operator.

Our main goal is the proof of some estimates on those holomorphy default 
operators, in the form of Theorem \ref{th:est_default}. As a first step, we define 
the space of symbols we are interested in, and prove some simple facts about this 
space.

\begin{definition}\label{def:symbols}
  Let $r:(0,\pi/2)\to\set R_+$ be a non-decreasing function, and for $\theta$ 
in $(0,\pi/2)$, let $U_{\theta,r(\theta)} = \{|z|>r(\theta), 
|\arg(z)|<\theta\}$ (see Figure \ref{fig:U_r_theta}). We note $S(r)$ the set 
of functions 
$\gamma$ from the union of the $U_{\theta,r(\theta)}$ to $\set C$ 
which are holomorphic and have sub-exponential growth on each 
$U_{\theta,r(\theta)}$, 
i.e. for each $\theta\in (0,\pi/2)$ and $\epsilon>0$, we have 
$p_{\theta,\epsilon}(\gamma) := \sup_{z\in 
U_{\theta,r(\theta)}} |\gamma(z)e^{-\epsilon|z|}| < +\infty$.
We endow $S(r)$ with the topology defined by the seminorms $p_{\theta,\epsilon}$ 
for all $\theta\in(0,\pi/2)$ and $\epsilon>0$.
\end{definition}

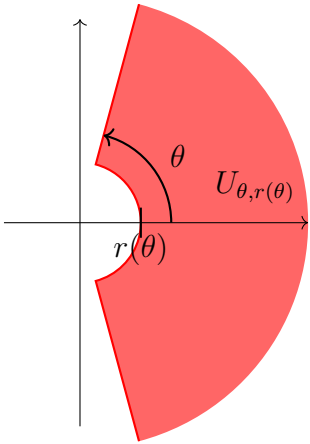
\begin{figure}
\begin{minipage}[c]{0.4\textwidth}
 \begin{center}
  \begin{tikzpicture}[baseline=(a),scale=1]
  \node (a) at (0,0) {};
  \path[fill=red!60] (75:3) -- (75:0.8)
    arc[radius=0.8, start angle= 75, end angle = -75] -- (-75:3)
    arc[radius=3, start angle=-75, end angle=75];
  \draw[red, thick] (75:3) -- (75:0.8)
    arc[radius=0.8, start angle= 75, end angle = -75] -- (-75:3);
  \node at (16:1.7) [right]{$U_{\theta,r(\theta)}$};
  \draw[->](-1,0) -- (3,0);
  \draw[->](0,-2.7) -- (0,2.7);
  \draw[thick] (0.8,-0.2) -- +(0,0.4);
  \draw[->, thick] (1.2,0) arc[start angle = 0, end angle = 75, radius = 1.2];
  \node at (30:1.2)[above right] {$\theta$};
  \node at (0.8,0)[below] {$r(\theta)$};
\end{tikzpicture}
 \end{center}
 \end{minipage}\hfill%
 \begin{minipage}[c]{0.6\textwidth}
\caption{An example of a set $U_{\theta,r(\theta)}$, whose union for 
$0<\theta<\pi/2$ is the domain of definition of functions in $S(r)$. The angle 
$\theta$ is allowed to be arbitrarily close to $\pi/2$, but then, the radius 
$r(\theta)$ of the disk we avoid can grow arbitrarily fast.}\label{fig:U_r_theta}
\end{minipage}
\end{figure}

From now on, when we write $S(r)$, it is implicitly assumed that $r$ is a 
non-decreasing function for $(0,\pi/2)$ to $\set R_+$.

\begin{ex}
  \begin{itemize}
    \item Every bounded holomorphic function on the half plane $\{\Re(z)\geq 0\}$ is in $S(0)$. For instance, $z\mapsto e^{-z}$ is in $S(0)$.
    \item Every polynomial is in $S(0)$.
    \item For all $s>0$, $z\mapsto z^s$ is in $S(0)$.
    \item More generally, if $\gamma$ is holomorphic on every domain $U_{\theta,r(\theta)}$ and has at most polynomial growth on those domains, $\gamma$ is in $S(r)$.
  \end{itemize}
\end{ex}

\begin{remark}
 The only values of a symbol $\gamma \in S(r)$ we are actually interested in are the 
values $\gamma(n)$ at the integers; the other values do not appear in the 
operator $H_\gamma: \sum a_n z^n \mapsto \sum \gamma(n) a_n z^n$. However, the 
holomorphic hypothesis, and hence the other values of $\gamma$, is quite essential 
for the proof of the estimate in Theorem~\ref{th:est_default}. It mainly appears to 
justify a change of integration path in the integral $\hat \gamma(\zeta) = 
\int_{0}^{+\infty} \gamma(x) e^{-ix\zeta}\diff x$ (see Propositions 
\ref{th:fourier_transform} and \ref{th:est_fourier}).

Even if they do not seem to play any role in the operator $H_\gamma$, the very fact 
that the values of $\gamma$ at non-integers exist impose some structure to the 
values $\gamma(n)$ at the integers. A structure we unfortunately have not been able 
to express in a different, more manageable way.
\end{remark}

Let us remind that if $\Omega$ is an open subset of $\set C$, $\mathcal O(\Omega)$ 
is 
the space of holomorphic functions in $\Omega$ with the topology of uniform 
convergence in every compact subset of $\Omega$. 

\begin{prop}\label{th:prop_symbols}The space $S(r)$ enjoys the following properties:
  \begin{itemize}
    \item The topology of $S(r)$ is stronger than the topology of uniform 
convergence on every compact.
    \item For all compact $K$ of $\bigcup U_{\theta,r(\theta)}$, and all $j\in \set N$, the seminorm $\gamma \mapsto |\gamma^{(j)}|_{L^\infty(K)}$ is continuous on $S(r)$.
    \item For all $z_0$ in the domain of definition of $\gamma$, the punctual evaluation at $z_0$, i.e. $\gamma\mapsto \gamma(z_0)$, is continuous on $S(r)$.
    \item The application $(\gamma_1,\gamma_2)\mapsto \gamma_1 \gamma_2$ is 
continuous from $S(r)\times S(r)$ to $S(r)$.
  \end{itemize}
\end{prop}
\begin{proof}
\begin{itemize}
  \item Let $K$ be a compact subset of $\bigcup_{0<\theta<\pi/2} 
U_{\theta,r(\theta)}$. By the Borel-Lebesgue property, there is a finite number of 
$\theta$ in $(0,\pi/2)$, say $\theta_1,\dotsc,\theta_k$ such that $K\subset 
\bigcup_{j=1}^k U_{\theta_k,r(\theta_k)}$. By noting $R=\sup_{z\in K}|z|$, we then 
have $|u|_{L^\infty(K)} \leq \sup_{1\leq j\leq k} p_{\theta_j,1}(\gamma)e^R$. This 
proves the first fact.
  
  \item We remind that if $\Omega$ is an open subset of $\set C$, $j$ is a natural number and $K$ a compact subset of $\Omega$ then Cauchy's integral formula implies that the seminorm on $\mathcal O(\Omega)$: $f\mapsto |f^{(j)}|_{L^\infty(K)}$ is continuous. Thus, the second point is a consequence of the first.
  
  \item Since $\{z_0\}$ is compact, the third point is a direct consequence of the second point (or the first).
  
  \item In order to prove the fourth point, we write for $z\in U_{\theta,r(\theta)}$: 
$|\gamma_1(z)\gamma_2(z)| \leq 
p_{\theta,\epsilon/2}(\gamma_1)p_{\theta,\epsilon/2}(\gamma_2) e^{\epsilon|z|}$, so 
$p_{\theta,\epsilon}(\gamma_1\gamma_2)\leq 
p_{\theta,\epsilon/2}(\gamma_1)p_{\theta,\epsilon/2}(\gamma_2)$.\qedhere
\end{itemize}
\end{proof}

\begin{prop}\label{th:change_domain}We have the following continuous injections 
between spaces $S(r)$:
  \begin{itemize}
    \item If $r'\geq r$, then denoting $U' = \bigcup U_{\theta,r'(\theta)}$, the restriction map $\gamma\in S(r) \mapsto \gamma_{|U'} \in S(r')$ is continuous.
    \item Let $\theta_0$ in $(0,\pi/2)$ and $a>r(\theta_0)$. Define $r'(\theta)$ by 
$r'(\theta) = 0$ if $|\theta|< \theta_0$ and $r'(\theta) = r(\theta)$ otherwise. Then 
$\gamma\in S(r)\mapsto \gamma(\cdot+a) \in S(r')$ is continuous.
  \end{itemize}
\end{prop}
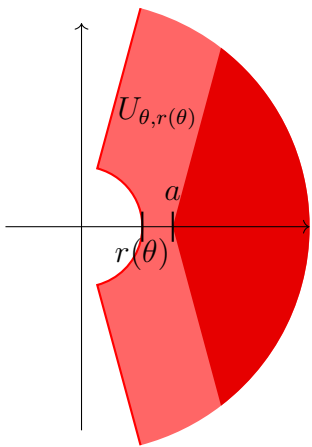
\begin{figure}
\begin{minipage}[c]{0.4\textwidth}
 \centering
 \begin{tikzpicture}[baseline=(a),scale=1]
  \path[fill=red!60] (75:3) -- (75:0.8)
    arc[radius=0.8, start angle= 75, end angle = -75] -- (-75:3)
    arc[radius=3, start angle=-75, end angle=75];
  \draw[red, thick] (75:3) -- (75:0.8)
    arc[radius=0.8, start angle= 75, end angle = -75] -- (-75:3);
  \node at (1,1.5){$U_{\theta,r(\theta)}$};  
  \path[fill = red!90!black] (52.3:3) -- (1.2,0) -- (-52.3:3) arc[start angle = 
-52.3, end angle = 52.3, radius = 3]; 
  \draw[->](-1,0) -- (3,0);
  \draw[->](0,-2.7) -- (0,2.7);
  \draw[thick] (0.8,-0.2) -- +(0,0.4) node[pos = 0.5, below]{$r(\theta)$};
  \draw[thick] (1.2,-0.2) -- +(0,0.4) node[above]{$a$};
\end{tikzpicture}
 \end{minipage}\hfill%
 \begin{minipage}[c]{0.6\textwidth}
 \caption{If $a> r(\theta)$, then the set $\{a+z,\break 
|\arg(z)|<\theta\}$ (in 
darker red) is a subset of $U_{\theta,r(\pi)}$.}\label{fig:change_domain}
\end{minipage}
\end{figure}

\begin{proof}
  \begin{itemize}
    \item For readability, let us write $U_\theta = U_{\theta,r(\theta)}$ and 
$U'_\theta = U_{\theta,r'(\theta)}$. To prove the first point, we simply remark that 
$r'\geq r$ implies $U'_{\theta}\subset U_{\theta}$, so, we have: 
$|\gamma(z)e^{-\epsilon|z|}|_{L^\infty(U'_\theta)} \leq 
|\gamma(z)e^{-\epsilon|z|}|_{L^\infty(U_\theta)}$.
    
    \item Looking at figure \ref{fig:change_domain} should convince us that it 
makes sense when looking at the domain of definition (we let the careful reader check 
it formally). The continuity is a consequence of: $|\gamma(z+a)e^{-\epsilon |z|}| 
\leq e^{\epsilon a}|\gamma(z+a)e^{-\epsilon|z+a|}|$.\qedhere
  \end{itemize}
\end{proof}

\subsection{Fourier transform of a symbol and convolution kernel}

The proof of the main estimate on holomorphy default operators relies on Poisson's 
summation formula applied to the sum $\sum \gamma(n) 
z^n$. In order to do that, we need some information on the Fourier transform of 
$\gamma$, the first of which being the existence of it.

We suppose in this subsection that for some $\theta_0$ in $(0,\pi/2)$, $r(\theta_0) = 
0$ (so that $r(\theta) = 0$ for $0<\theta\leq \theta_0$). Then we define the 
Fourier transform $\hat \gamma$ of $\gamma$ for $\xi$ with 
negative imaginary part by $\hat\gamma(\xi) = \int_{0}^{+\infty} \gamma(x) 
e^{-ix\xi}\diff x$. We first prove that this Fourier transform can be 
extended on a bigger domain than the lower half-plane, then, assuming some 
regularity at $0$, we prove an estimate on it.
\begin{prop}\label{th:fourier_transform}
  Let $\gamma$ in $S(r)$. The Fourier transform $\hat \gamma$ of $\gamma$, which is 
holomorphic on  $\{\Im(\xi)<0\}$, can be holomorphically extended on $\set 
C \setminus i\set [0,+\infty)$.
\end{prop}
\begin{proof}
  Let $\phi$ in $(0,\pi/2)$, let $\theta$ in $(\phi,\pi/2)$ and $r_1>r(\theta)$. We 
make a change of contour in the integral defining $\hat \gamma(\xi)$: let $c$ the 
path $[0,r_1]\cup {\{r_1e^{i\varphi},-\phi\leq \varphi\leq 0\}}\cup e^{-i\phi} 
[r_1,+\infty)$ (see figure \ref{fig:fourier1}). We have for $\xi$ in $\{\Im(\xi)<0\} 
\cap e^{i\phi}\{\Im(\xi)<0\}$:
  \begin{align*}
\hat\gamma(\xi) &=\int_{0}^{+\infty} \gamma(x)e^{-ix\xi}\diff x\\
  &= \int_c \gamma(z)e^{-iz\xi}\diff z\\
  &= \int_{0}^{r_1}\gamma(x)e^{-ix\xi}\diff x +
     \int_{0}^{-\phi} \gamma(r_1 e^{it})e^{-ie^{it}\xi}ir_1 e^{it}\diff t\\
  &\qquad\qquad\qquad   +\int_{r_1}^{+\infty} \gamma(e^{-i\phi}x)e^{-ie^{-i\phi}\xi 
x}e^{-i\phi}\diff x.
\end{align*}

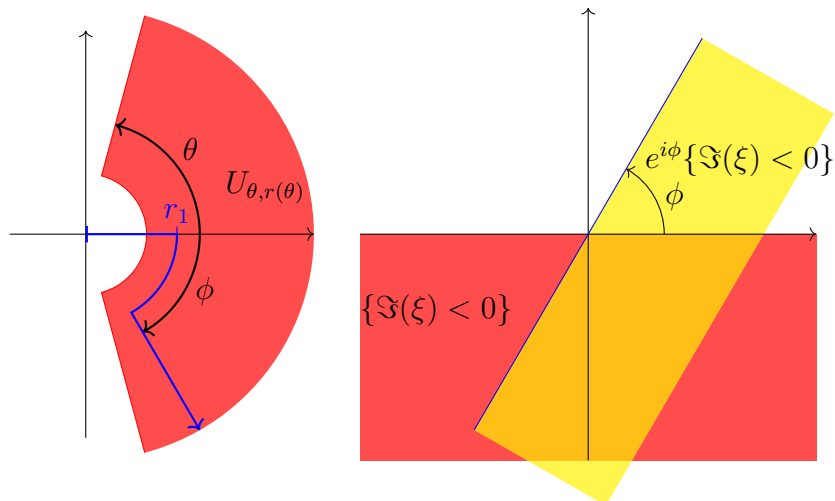
\begin{figure}[hbt]
 \begin{center}
  \begin{tikzpicture}[baseline=(a),scale=1]
  \node (a) at (0,0) {};
  \path[fill=red!70] (75:3) -- (75:0.8)
    arc[radius=0.8, start angle= 75, end angle = -75] -- (-75:3)
    arc[radius=3, start angle=-75, end angle=75];
  \draw[red] (75:3) -- (75:0.8)
    arc[radius=0.8, start angle= 75, end angle = -75] -- (-75:3);
  \node at (20:1.8) [right]{$U_{\theta,r(\theta)}$};
  \draw[->](-1,0) -- (3,0);
  \draw[->](0,-2.7) -- (0,2.7);
  \draw[thick, ->](1.5,0) arc[radius=1.5, start angle = 0, delta angle=-60];
  \node at (-30:1.5) [right]{$\phi$};
  \draw[thick, ->](1.5,0) arc[radius=1.5, start angle = 0, delta angle=+75];
  \node at (45:1.6) [right]{$\theta$};
  \draw[blue,|->, thick] (0.,0) -- (1.2,0)  arc[radius=1.2, start angle= 0, end 
  angle = -60] -- (-60:3);
  \draw[blue] (1.2,0.1) -- ++(0,-0.2);
  \node at (1.2,0)[above,blue] {$r_1$};
\end{tikzpicture}
\begin{tikzpicture}[baseline=(a),scale=1]
  \node (a) at (0,0) {};
  \path[fill=red!100, fill opacity=0.7] (-3,0) rectangle (3,-3);
  \path[fill=yellow!100, fill opacity=0.7, rotate=60] (-3,0) rectangle (3,-2);
  \draw[->](-3,0) -- (3,0);
  \draw[->](0,-3) -- (0,3);
  \draw[blue] (240:3) -- (60:3);
  \draw[->](1,0) arc[radius=1, start angle = 0, delta angle=60];
  \node at (30:1) [right]{$\phi$};
  \path (-2,-1) node{$\{\Im(\xi)<0\}$};
  \path (2,1) node{$e^{i\phi}\{\Im(\xi)<0\}$};
\end{tikzpicture}
 \end{center}
 \caption{In the left figure: in red, a part of the domain of definition of 
$\gamma$, and in blue, an integration path that allows us to extend $\hat 
\gamma$. In the right  figure: in red, the a priori domain of definition of 
$\hat \gamma$, in yellow, the domain we extend $\hat \gamma$ to, when choosing 
the blue integration path of the left figure.}\label{fig:fourier1}
\end{figure}

The first two terms can be extended holomorphically on $\set C$, while the third can be extended holomorphically on $e^{i\phi}\{\Im(\xi)<0\}$. So, taking $\phi\to\pi/2$, $\hat \gamma$ can be extended holomorphically on $\{\Im(\xi)<0\} \cup i\{\Im(\xi)<0\}$. By taking the path $c'$ the symmetric of $c$ with respect to the real line, $\hat \gamma$ can also be extended holomorphically on $\{\Im(\xi)<0\} \cup -i \{\Im(\xi)<0\}$.
\end{proof}

\begin{prop}\label{th:est_fourier}
  Let $\epsilon > 0$. There exists $C>0$ and $\eta >0$ such that for all $\gamma$ in $S(r)$ satisfying $p(\gamma) := \sup_{|z|<1,|\arg(z)|<\theta_0} \frac{|\gamma(z)|}{|z|} < +\infty$ and for all $\xi$ in $\{-ire^{i\theta}, r>\epsilon, |\theta|<2\theta_0\}$ (see figure \ref{fig:fourier}):
  \[
    |\hat\gamma(\xi)| \leq C(p(\gamma)+p_{\theta_0,\eta}(\gamma))|\xi|^{-2}.
  \]
\end{prop}
\begin{proof}
  The proof is mostly redoing the calculation of the proof of the previous 
proposition, but this time using the increased regularity (at $0$) to get the stated 
estimate.
  
  Let $\xi = -ire^{i2\theta}$ with $|\theta|<\theta_0$ and $r>\epsilon$. Thanks to 
the proof of the previous proposition, we have $\hat \gamma(\xi) = 
\int_{e^{-i\theta}\set R_+} \gamma(z) e^{-iz\xi} \diff z = e^{-i\theta} 
\int_0^{+\infty} \gamma(e^{-i\theta}x) e^{-e^{i\theta}xr}\diff x$. We then write 
$|\hat \gamma(\xi)| \leq \int_0^1 p(\gamma) xe^{-\cos(\theta) rx}\diff x + 
\int_1^{+\infty} p_{\theta_0,\eta}(\gamma) e^{\eta x} e^{-\cos(\theta)r x}\diff x$, 
which is true for all $\eta>0$.
\begin{figure}[ht]
 \begin{minipage}[c]{0.6\textwidth}
  \begin{center}
  \begin{tikzpicture}[scale=1]
  \fill[yellow!60] (120:0.5) -- (120:3) 
    arc[radius = 3, start angle = 120, delta angle = 300]
    -- (60:0.5) arc[radius=0.5,start angle=60, delta angle=-300] -- cycle;
  \fill[red!70] (-165:0.5) -- (-165:3) 
    arc[radius = 3, start angle = -165, delta angle = 150]
    -- (-15:0.5) arc[radius = 0.5, start angle = -15, delta angle=-150] -- cycle;
  
  \draw[thick,->] (0,-1) arc[radius = 1, start angle = -90, delta angle = 75];
  \node[right] at (-65:1.1){$\theta_0$};
  \draw[thick,->] (0,-0.8) arc[radius = 0.8, start angle = -90, delta angle = 150];
  \node[right] at (30:0.8){$2\theta_0$};
  
  \draw (40:2)  node[above right]{$\xi = -ire^{i2\theta}$} \cross{0.2};
  \draw (-25:2)  node[below right]{$-ire^{i\theta}$} \cross{0.2};
  
  \draw[<->, thick] (0,0) -- (-110:0.5) node[pos=0.4, left]{$\epsilon$};
  
  \draw[->] (-3,0) -- (3,0);
  \draw[->] (0,-3) -- (0,3);
  \draw (-15:0.5) -- (-15:3)
	(-165:0.5) -- (-165:3)
	(60:0.5) -- (60:3)
	(120:0.5) -- (120:3)
	(120:0.5) arc[radius=0.5, start angle = 120, delta angle = 300];
\end{tikzpicture}
  \end{center}
 \end{minipage}\hfill%
 \begin{minipage}[c]{0.4\textwidth}
 \caption{The sub-exponential growth of $\gamma$ gives us an estimate on $\hat 
 \gamma$  on the red domain, and a change of integration path allows us to extend 
 this  estimate  on the yellow domain.}\label{fig:fourier}
 \end{minipage}
\end{figure}
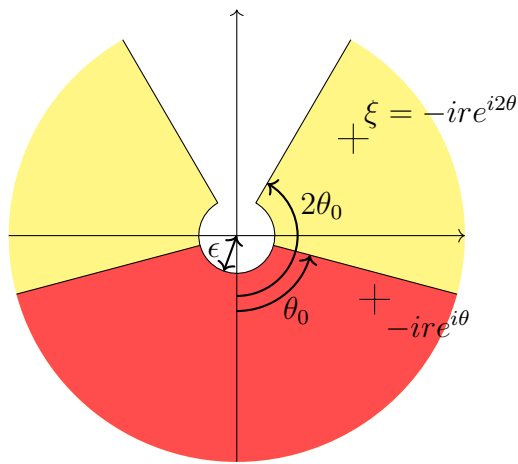

  For the first term of the right hand side, we make the change of variables $x' = 
\cos(\theta) r x$, so that $\int_0^1 xe^{-\cos(\theta)rx}\diff x = 
\frac1{(r\cos(\theta))^2} \int_0^{r\cos(\theta)} x' e^{-x'}\diff x \leq 
\frac1{(r\cos(\theta))^2} \Gamma(2)$.
  
  For the second term of the right hand side, we have $\int_1^{+\infty} e^{x(\eta - \cos(\theta)r)}\diff x = \frac1{\cos(\theta)r - \eta} e^{\eta - \cos(\theta)r}$ as long as $\eta < \cos(\theta)r$. We then choose $\eta = \frac\epsilon2\cos(\theta_0)$ so that $|\theta|<\theta_0$ and $r>\epsilon$ implies $\cos(\theta)r - \eta > \cos(\theta_0) \epsilon - \frac\epsilon2\cos(\theta_0) = \frac12\cos(\theta_0) \epsilon$. So $\int_1^{+\infty} e^{x(\eta-\cos(\theta)r)}\diff x \leq \frac2{\epsilon\cos(\theta_0)}e^\eta e^{-r\cos(\theta_0)}$. So, writing $c = \sup_{t>0}(t^2e^{-t})$ and $C_2 = \frac{2ce^\eta}{\epsilon\cos(\theta_0)}$, we have $\int_0^{+\infty} e^{x(\eta-\cos(\theta)r)}\diff x \leq C_2 \frac1{(r\cos(\theta_0))^2}$.
  
  Combining these two inequalities, we have:
  \[|\hat \gamma(\xi)|\leq (\Gamma(2) p(\gamma) + C_2 
p_{\theta_0,\eta}(\gamma))\frac{1}{\cos(\theta_0)^2} r^{-2}\qedhere.\]
\end{proof}

With the previous two properties, we can prove the main tool for establishing estimates on holomorphy default operators:
\begin{prop}\label{th:conv_kernel}
  Let $\gamma$ in $S(r)$ and $K_\gamma$ the function defined by $K_\gamma(z) = \sum 
\gamma(n) z^n$. Then $K_\gamma$ admits an holomorphic extension to $\set C \setminus 
[1,+\infty[$. Moreover, the map $\gamma\in S(r) \mapsto K_\gamma \in \mathcal 
O(\set C\setminus[1,+\infty[)$ is continuous.
\end{prop}

\begin{remark}
This theorem was already essentially proved by Lindelöf \cite{lindelof1905calcul} in 
the special case $r(\theta) = \frac{r_0}{\cos(\theta)}$, that is when the domain of 
definition of $\gamma$ is the half-plane $\{\Re(z)>r_0\}$, and the case of a general 
$r$ was proved by Arakelyan \cite{_efficient_1985}. Our method is different than 
the previous 
two references, and, most importantly, we prove the continuity of the analytic 
continuation with respect to the topology of $S(r)$.
\end{remark}

\begin{proof}
Let $G$ be a connected relatively compact open subset of $\set C \setminus [1,+\infty)$.  We suppose without loss of generality that $G$ intersects the unit disk $D(0,1)$. We want to show that $K_\gamma$ can be extended to a bounded holomorphic function on $G$, and that this extension depends continuously on $\gamma$ for the topology of uniform convergence on $G$.

First, we reduce the problem to a case where we can use the previous estimate on the Fourier transform of symbols, by defining $\tilde \gamma(z) = \gamma(z+n_1) - \gamma(n_1)$ with $n_1$ large enough. We can explicitly compute $K_\gamma$ from $K_{\tilde \gamma}$, so we focus on the later, and apply Poisson summation formula to the sum defining $K_{\tilde \gamma}$, the estimate on $\hat{\tilde \gamma}$ allowing us to holomorphically extend the sum.

\paragraph{Choice of $n_1$}
The Poisson summation formula will involve terms of the form $\hat{\tilde\gamma}(i\ln(\zeta) + 2\pi k)$, so we let $F = \{\xi \in \set C, e^{-i\xi}\in G\}$. For all $\zeta$ in $\set C$, $\zeta\in G$ is equivalent to $i\ln(\zeta)\in F$, whatever the determination of the logarithm.
\begin{figure}[ht]
  \begin{center}
    \definecolor{cff0000}{RGB}{255,0,0}
\definecolor{c000067}{RGB}{0,0,103}

\begin{tikzpicture}[y=0.80pt, x=0.80pt, yscale=-1.000000, xscale=1.000000, inner sep=0pt, outer sep=0pt]
\path[fill=cff0000,line join=miter,line cap=butt,fill opacity=0.662,even odd
  rule,line width=0.800pt] (205.0000,122.3622) -- (205.0000,198.0711) --
  (415.0000,198.0711) -- (415.0000,122.3622)(345.0000,122.3622) .. controls
  (340.0000,122.3622) and (330.0000,167.3622) .. (320.0000,167.3622) .. controls
  (310.0000,167.3622) and (305.0000,177.3622) .. (295.0000,162.3622) .. controls
  (285.0000,147.3622) and (280.0000,122.3622) ..
  (275.0000,122.3622)(275.0217,122.3622) .. controls (270.0217,122.3622) and
  (260.0217,167.3622) .. (250.0217,167.3622) .. controls (240.0217,167.3622) and
  (235.0217,177.3622) .. (225.0217,162.3622) .. controls (215.0217,147.3622) and
  (210.0217,122.3622) .. (205.0217,122.3622)(415.0000,122.3622) .. controls
  (410.0000,122.3622) and (400.0000,167.3622) .. (390.0000,167.3622) .. controls
  (380.0000,167.3622) and (375.0000,177.3622) .. (365.0000,162.3622) .. controls
  (355.0000,147.3622) and (350.0000,122.3622) .. (345.0000,122.3622);
\path[draw=black,line join=miter,line cap=butt,even odd rule,line width=0.800pt]
  (195.5162,162.5455) -- (425.2699,162.4539);
\path[draw=black,line join=miter,line cap=butt,even odd rule,line width=0.800pt]
  (309.9467,212.3958) -- (309.9816,87.3854);
\path[draw=black,miter limit=4.00,line width=0.800pt]
  (310.1573,137.3627)arc(270.360:296.565:25.000000 and 25.000);
\path[xscale=0.711,yscale=1.406,fill=black,line join=miter,line cap=butt,line
  width=0.800pt] (440.6323,95.4886) node[above right] (text16091) {$\phi$};
\path[draw=black,line join=miter,line cap=butt,even odd rule,line width=0.800pt]
  (379.9993,159.3622) -- (379.9993,165.3622);
\path[xscale=0.713,yscale=1.403,fill=black,line join=miter,line cap=butt,line
  width=0.800pt] (532.2780,124.1620) node[above right] (text16097) {$2\pi$};
\path[draw=black,line join=miter,line cap=butt,even odd rule,line width=0.800pt]
  (240.2337,158.7313) -- (240.2337,164.7313);
\path[xscale=0.713,yscale=1.403,fill=black,line join=miter,line cap=butt,line
  width=0.800pt] (336.7041,123.1897) node[above right] (text16097-0) {$-2\pi$};
\path[xscale=0.887,yscale=1.128,fill=black,line join=miter,line cap=butt,line
  width=0.800pt] (307.8417,165.3881) node[above right] (text19778) {$F$};
\path[draw=black,fill=cff0000,line join=miter,line cap=butt,miter
  limit=4.00,fill opacity=0.662,even odd rule,line width=0.320pt]
  (536.8809,160.8484) .. controls (546.8809,145.8484) and (544.0293,136.5816) ..
  (539.0293,131.5816) .. controls (534.0293,126.5816) and (510.4146,131.2150) ..
  (500.4146,136.2150) .. controls (494.0900,139.3773) and (480.4146,151.2150) ..
  (480.4146,161.2150) .. controls (480.4146,171.2150) and (485.4146,181.2150) ..
  (490.4146,186.2150) .. controls (500.4146,196.2150) and (536.0312,191.7702) ..
  (539.4258,188.0991) .. controls (550.4146,176.2150) and (530.4146,171.2150) ..
  (536.8809,160.8484) -- cycle;
\path[draw=black,line join=miter,line cap=butt,even odd rule,line width=0.800pt]
  (432.2472,161.7457) -- (594.4223,162.5863);
\path[draw=black,line join=miter,line cap=butt,even odd rule,line width=0.800pt]
  (520.3410,217.2695) -- (520.4929,82.5774);
\path[draw=black,line join=miter,line cap=butt,even odd rule,line width=0.800pt]
  (542.8894,159.0870) -- (542.8894,165.0870);
\path[xscale=0.783,yscale=1.277,fill=black,line join=miter,line cap=butt,line
  width=0.800pt] (692.3909,136.8159) node[above right] (text10280) {$1$};
\path[xscale=0.574,yscale=1.741,fill=black,line join=miter,line cap=butt,line
  width=0.800pt] (862.9313,91.1825) node[above right] (text10308) {$G$};
\path[xscale=0.642,yscale=1.557,fill=black,line join=miter,line cap=butt,line
  width=0.800pt] (898.3008,104.3248) node[above right] (text10333) {$\phi$};
\path[draw=black,miter limit=4.00,fill opacity=0.407,line width=0.800pt]
  (571.1008,151.3647)arc(-19.885:1.332:30.000);
\path[draw=black,line join=round,line cap=butt,even odd rule,line width=0.800pt]
  (591.0458,165.5455) -- (594.0000,162.3622) -- (591.0000,159.3622);
\path[draw=black,line join=round,line cap=butt,even odd rule,line width=0.800pt]
  (422.6954,165.6371) -- (425.6496,162.4538) -- (422.6496,159.4538);
\path[draw=black,line join=round,line cap=butt,even odd rule,line width=0.800pt]
  (306.8156,90.0060) -- (309.9989,87.0518) -- (312.9989,90.0518);
\path[draw=black,line join=round,line cap=butt,even odd rule,line width=0.800pt]
  (517.3221,85.8847) -- (520.5054,82.9305) -- (523.5054,85.9305);
\path[draw=black,line join=round,miter limit=4.00,fill opacity=0.662,line
  width=0.800pt]
  (321.2103,140.0165)arc(-63.358:13.298:25.000)arc(13.298:89.955:25.000);
\path[xscale=0.433,yscale=2.310,fill=black,line join=miter,line cap=butt,line
  width=0.800pt] (774.0784,67.1407) node[above right] (text5723) {$2\theta_0$};
\path[draw=black,line join=round,line cap=butt,even odd rule,line width=0.800pt]
  (319.3909,135.9346) -- (320.8781,140.0149) -- (316.9489,141.6154);
\path[draw=black,line join=round,line cap=butt,even odd rule,line width=0.800pt]
  (322.6669,144.2534) -- (321.3000,140.1313) -- (325.2744,138.6467);
\path[draw=c000067,line join=round,miter limit=4.00,fill opacity=0.662,even odd
  rule,line width=0.800pt] (345.0000,92.3622) -- (312.2460,158.0026) .. controls
  (314.3167,158.9338) and (315.4022,161.2654) .. (314.8634,163.5230) .. controls
  (314.3245,165.7806) and (312.3029,167.3705) .. (309.9819,167.3621) .. controls
  (307.6609,167.3541) and (305.6509,165.7492) .. (305.1284,163.4878) .. controls
  (304.6058,161.2264) and (305.7082,158.9027) .. (307.7702,158.0258) --
  (275.0000,92.3622);
\path[draw=c000067,dash pattern=on 0.32pt off 0.32pt,line join=round,miter
  limit=4.00,fill opacity=0.662,even odd rule,line width=0.320pt]
  (275.4001,92.3622) -- (242.6462,158.0026) .. controls (244.7168,158.9338) and
  (245.8024,161.2654) .. (245.2635,163.5230) .. controls (244.7246,165.7806) and
  (242.7030,167.3705) .. (240.3821,167.3621) .. controls (238.0611,167.3541) and
  (236.0510,165.7492) .. (235.5285,163.4878) .. controls (235.0060,161.2264) and
  (236.1084,158.9027) .. (238.1703,158.0258) -- (205.4001,92.3622);
\path[draw=c000067,dash pattern=on 0.32pt off 0.32pt,line join=round,miter
  limit=4.00,fill opacity=0.662,even odd rule,line width=0.320pt]
  (415.0158,92.3622) -- (382.2619,158.0026) .. controls (384.3325,158.9338) and
  (385.4181,161.2654) .. (384.8792,163.5230) .. controls (384.3404,165.7806) and
  (382.3188,167.3705) .. (379.9978,167.3621) .. controls (377.6768,167.3541) and
  (375.6667,165.7492) .. (375.1442,163.4878) .. controls (374.6217,161.2264) and
  (375.7241,158.9027) .. (377.7860,158.0258) -- (345.0158,92.3622);
\path[draw=black,line join=round,line cap=butt,even odd rule,line width=0.800pt]
  (569.2214,155.5926) -- (570.9984,151.6299) -- (574.9317,153.2202);
\path[draw=c000067,line join=round,miter limit=4.00,fill opacity=0.662,even odd
  rule,line width=0.800pt] (496.6882,201.7801) .. controls (532.4563,216.6505)
  and (564.8229,199.6209) .. (564.8727,179.4185) .. controls (564.8727,179.4185)
  and (566.3223,170.7837) .. (547.7152,164.6736) .. controls (547.0000,166.3622)
  and (544.7783,167.7820) .. (542.4880,167.4288) .. controls (540.1977,167.0756)
  and (538.4502,165.1950) .. (538.2659,162.8850) .. controls (538.0815,160.5749)
  and (539.5086,158.4409) .. (541.7139,157.7290) .. controls (543.9192,157.0171)
  and (546.3247,157.9138) .. (547.4840,159.8841) .. controls (567.0000,152.3622)
  and (564.6084,146.4200) .. (564.5766,144.8492) .. controls (564.1674,124.6509)
  and (532.0293,108.0592) .. (496.2611,122.9296) .. controls (474.5947,131.9372)
  and (446.2527,161.8245) .. (446.2527,161.8245) .. controls (446.2527,161.8245)
  and (475.0218,192.7724) .. (496.6882,201.7801) -- cycle;
\path[draw=black,line join=miter,line cap=butt,miter limit=4.00,even odd
  rule,line width=0.320pt] (542.4519,162.4804) -- (592.4519,142.4804);
\path[draw=black,line join=miter,line cap=butt,miter limit=4.00,even odd
  rule,line width=0.320pt] (542.7957,162.2744) -- (592.7957,182.2744);

\end{tikzpicture}
  \end{center}
  \caption{Left figure: in red, the domain $F$, in plain blue, the boundary of $\{-ire^{i\theta},r>\epsilon, |\theta|\leq 2\theta_0\}$ and in dotted blue, the boundary of the $2\pi$-periodic version of the previous domain. Right figure: in red, the domain $G = e^{-iF}$, and in blue, the boundary of $\{e^\xi, |\xi|>\epsilon,|\arg(\xi)|>\phi\}$.}\label{fig:n_1}
\end{figure}
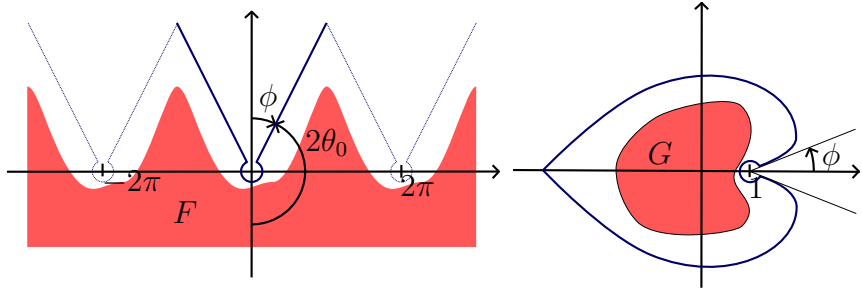

Since $G$ is relatively compact on $\set C\setminus [1,+\infty)$, there exists $\phi$ 
in $(0,\pi)$ and $\epsilon>0$ such that $G\subset \set C\setminus \{e^\xi,|\xi| \leq 
\epsilon \text{ or } |\arg(\xi)|\leq\phi\}$. Then, noting $\theta_0 = 
\frac12(\pi-\phi)$, $F$ is a subset of $\{-ire^{i\theta}, r>\epsilon,|\theta|< 
2\theta_0\}$ (see figure \ref{fig:n_1}). Let $n_1$ be a natural number greater than 
$r(\theta_0)$, for instance $n_1 = \lfloor r(\theta_0) \rfloor +1$, let $\tilde r: 
(0,\pi/2)\to \set R_+$ be defined by $\tilde r(\theta) = 0$ for 
$0<\theta\leq\theta_0$ and $\tilde r(\theta) = r(\theta)$ for $\theta> \theta_0$, and 
let $\tilde \gamma$ be defined by $\tilde\gamma(z) = \gamma(z+n_1) - \gamma(n_1)$.

According to the second point of Proposition \ref{th:change_domain}, $\tilde 
\gamma$ is in $S(\tilde r)$ and depends continuously on $\gamma$. Moreover, we 
have for $z$ in $\{|z|<1, |\arg(z)|<\theta_0\}$, $|\tilde\gamma(z)| \leq 
\sup_{t\in[n_1,z+n_1]}|\gamma'(t)| |z|$, so, if we define $p(\tilde \gamma)$ as 
in the previous proposition by $p(\tilde \gamma) = \sup_{|z|<1,|\arg(z)|<\theta_0} 
\frac{|\tilde \gamma (z)|}{|z|}$, we have $p(\tilde \gamma) \leq \sup_{|z|\leq 
1, |\arg(z)|\leq \theta_0} |\gamma'(z+n_1)|$, which is finite since the subset 
$\{z+n_1, |z|\leq 1, |\arg(z)|\leq \theta_0\}$ is compact in $\bigcup U_\theta$, 
and thanks to the second point of Proposition \ref{th:prop_symbols}, $\gamma\mapsto 
p(\tilde \gamma)$ is continuous.

\paragraph{Relation between $K_\gamma$ and $K_{\tilde\gamma}$}
We have for all $\zeta$ in the unit disk:
\begin{align}
 K_\gamma(\zeta) &= \sum_{n>r(0)} \gamma(n)\zeta^n \notag\\
  &= \sum_{r(0)<n<n_1} \gamma(n)\zeta^n + \zeta^{n_1}\left(
     \gamma(n_1)\sum_{n\geq 0} \zeta^n + \sum_{n\geq 0} \tilde \gamma(n)\zeta^n 
     \right)\notag\\
  &= \sum_{r(0)<n<n_1}\gamma(n)\zeta^n + \gamma(n_1)\frac{\zeta^{n_1}}{1-\zeta} 
+ \zeta^{n_1}K_{\tilde \gamma}(\zeta).\label{eq:gamma_nu}
\end{align}

So, if we prove that $K_{\tilde \gamma}$ extends holomorphically to $G$ and that the extension depends continuously on $\tilde \gamma$ in the topology of uniform convergence on $G$, we will have proved the same for $K_\gamma$.

\paragraph{Poisson summation formula and holomorphic extension} We have by definition of $K_{\tilde \gamma}$, for all $|\zeta|<1$: $K_{\tilde\gamma}(\zeta) = \sum_{n> 0} \tilde \gamma(n) \zeta^n$. So, the Poisson summation formula implies that for all $|\zeta|<1$:

\[K_{\tilde\gamma}(\zeta) = 2\pi \sum_{k\in \set Z} \widehat{\zeta^x 
\tilde\gamma(x)}(2\pi k) = 2\pi \sum_{k\in \set Z}\hat{\tilde\gamma}(i\ln \zeta + 
2\pi k).\]

Let us recall that $F = \{\xi, e^{-i\xi}\in G\}$ is a subset of $\{-ire^{i\theta}, 
r>\epsilon, |\theta|<2\theta_0\}$, and let us remark that it's a 
$2\pi$\nobreakdash-periodic domain, so if $z$ is in $F$, then for all $k\in \set Z$, 
$|z+2\pi k|>\epsilon$. So the estimate of Fourier transform of symbols (Proposition 
\ref{th:est_fourier}) implies that the sum $ k_{\tilde \gamma}(z) := \sum_{k\in \set 
Z} \hat{\tilde\gamma}(z + 2\pi k)$ converges, and satisfies $|k_{\tilde 
\gamma}(z)|\leq C(p(\tilde \gamma) + p_{\theta_0,\eta}(\tilde \gamma))\sum_{k\in \set 
Z}|z+2\pi k|^{-2}\leq C'_\epsilon (p(\tilde\gamma) + p_{\theta_0,\eta}(\tilde 
\gamma))$. Moreover, this sum converges uniformly in $z\in F$, so the limit function 
$k_{\tilde \gamma}$ is holomorphic, and depends continuously on $\tilde \gamma\in 
S(\tilde r)$.

Since we have $K_{\tilde \gamma}(\zeta) = k_{\tilde \gamma}(i\ln \zeta)$, $K_{\tilde 
\gamma}$ extends holomorphically on $G\setminus [0,+\infty)$. But we already knew 
that $K_{\tilde \gamma}$ is holomorphic on the unit disk, so $K_{\tilde \gamma}$ is 
holomorphic in $G$. Moreover, since $k_{\tilde \gamma}$ depends continuously on 
$\tilde \gamma$, $K_{\tilde \gamma}$ also depends continuously on $\tilde \gamma$. 
This completes the proof of the proposition.
\end{proof}

\subsection{Proof of the estimate for the holomorphy default 
operators}
Before stating the estimates for holomorphy default operators, let us define a few 
notations. Let $r$ be an non-decreasing function from $(0,\pi/2)$ to $\set R_+$. We 
note $r(0) = \inf_{\theta\in(0,\pi/2)}r(\theta)$. Let $\mathcal O_{r(0)}$ be the 
closed subspace of $\mathcal O(\set C)$ of entire functions of the form 
$\sum_{n>r(0)} a_n z^n$, i.e. $\mathcal O_{r(0)} = \{f\in \mathcal O(\set C), \forall 
0\leq j\leq r(0), f^{(j)}(0) = 0\}$. If we endow $\mathcal O_{r(0)}$ with the 
$L^\infty(U)$ norm for some open bounded subset $U$ of $\set C$, we will note this 
space $\mathcal O_{r(0)}^\infty(U)$.

\begin{theorem}\label{th:est_default}
Let $r:(0,\pi/2)\to \set R_+$ be a non-decreasing function. Let $\gamma$ in $S(r)$ 
and $H_\gamma$ the operator on polynomials with $\lfloor r(0)\rfloor$ vanishing 
derivatives at $0$, defined by: 
\[H_\gamma\bigg(\sum_{n>r(0)} a_n z^n\bigg) = \sum_{n>r(0)} \gamma(n) a_n z^n.\]
Let $U$ be an open bounded domain, star shaped with respect to $0$. Let $\delta>0$ 
and $U^\delta = \{z\in \set C,\distance(z,U)<\delta\}$. Then there exists $C>0$ such 
that for all polynomials $f$ with vanishing derivatives of order up to 
$\lfloor r(0)\rfloor$:
\[|H_\gamma(f)|_{L^\infty(U)}\leq C|f|_{L^\infty(U^\delta)}.\]

Moreover, the constant $C$ above can be chosen continuously in $\gamma\in S(r)$: the 
map $\gamma\in S(r) \mapsto H_\gamma$ is continuous from $S(r)$ to $\mathcal 
L\big(\mathcal O^\infty_{r(0)}(U^\delta),\, \mathcal O^\infty_{r(0)}(U)\big)$.
\end{theorem}

Before proving the theorem, let us remark that the sub-exponential growth of 
$\gamma$ 
implies that $H_\gamma$ do maps $\mathcal O_{r(0)}$ to $\mathcal O_{r(0)}$, so the 
theorem actually makes sense.

\begin{proof}
  Let $R>0$ large enough so that $\bar U \subset D(0,R)$. If $f=\sum a_n z^n$ is an entire function, we have $a_n = \frac1{2i\pi} \oint_{\partial D(0,R)} \frac{f(\zeta)}{\zeta^{n+1}} \diff \zeta$, so, for $z$ in $U$, we have:
  
  \begin{align*}H_\gamma(f)(z) &=
  \sum_{n} \gamma_n \frac1{2i\pi} 
   \oint_{\partial D(0,R)} \frac{f(\zeta)}{\zeta^{n+1}}z^n \diff \zeta\\
   &= \oint_{\partial D(0,R)} \frac1{2i\pi \zeta} K_\gamma\left(\frac z\zeta\right) 
f(\zeta)\diff \zeta.
  \end{align*}
  
  We want to change the integration path for one that is closer to $U$. For any 
closed curve $c$ around $U$, since $U$ is star-shaped with respect to $0$, for any 
$z\in U$ and $\zeta \in c$, we never have $z/\zeta \in [1,+\infty)$. So, the subset 
$\{z/\zeta, z\in U, \zeta\in c\}$ is a compact subset of $\set C\setminus 
[1,+\infty)$, and according to the previous proposition, $M_c(\gamma):=\sup_{z\in U, 
\zeta \in c} \big|K_\gamma\big(\frac z\zeta\big)\big|$ is finite and depends 
continuously on $\gamma \in S(r)$.
  
  So, we have for $z \in U$, $|H_{\gamma}(z)|\leq \sup_{\zeta\in c}\frac1{2\pi|\zeta|} M_c(\gamma) \sup_c|f|$. Since we can choose $c$ as close as we want to $U$, this proves the theorem.
\end{proof}

\begin{remark}
  Actually, the theorem we proved is the following: if $(\gamma_n)$ is a sequence of 
complex numbers such that the entire series $\sum_{n\geq 0}\gamma_n \zeta^n$ has 
non-zero convergence radius and that $K_\gamma(\zeta) := \sum_{n\geq 0} \gamma_n 
\zeta^n$ admits an holomorphic extension on $\set C \setminus [1,+\infty)$, then, for 
all domain $U$ satisfying the hypotheses of the theorem, and for all $\delta>0$, 
there exists $C>0$ such that for all entire functions $f$, 
$|H_\gamma(f)|_{L^\infty(U)}\leq C|f|_{L^\infty(U^\delta)}$. Moreover, $C$ can be 
chosen continuously in $K_\gamma$ (for the topology of uniform convergence on every 
compact).
\end{remark}

\section{Spectral analysis of the Fourier components}\label{sec:spectral}
\subsection{Introduction}
In this section, we prove estimates on the first eigenvalue
$\lambda_\alpha$ of $-\partial_x^2 + (\alpha x)^2$ on $(-1,1)$ with Dirichlet 
boundary conditions, and on its
associated eigenfunction. Let us recall some facts already mentioned by K. 
Beauchard, P. Cannarsa and R. Guglielmi \cite{beauchard_null_2014-1}, which are 
proved 
thanks to Sturm-Liouville's theory:
\begin{prop}\label{th:basic_eigen}
  Let $\alpha$ be a real number. The (unbounded) operator $P_\alpha = {-\partial_x^2 
+ (\alpha x)^2}$ on $L^2$ (with domain $H^1_0(-1,1)\cap H^2(-1,1)$) admits an 
orthonormal basis $(v_{\alpha k})_{k\geq 0}$ of eigenvectors, with the associated 
eigenvalues sequence $(\lambda_{\alpha k})_{k\geq0}$ being non-decreasing and 
tending to $+\infty$ as $k\to+\infty$. Moreover, the first eigenvalue 
$\lambda_\alpha = \lambda_{\alpha0}$ is simple, greater than $|\alpha|$, 
and we have $\lambda_\alpha \sim_{\alpha\to+\infty} \alpha$. Finally, the associated 
eigenvector $v_\alpha = v_{\alpha0}$ is even, positive on $(-1,1)$ 
and non-increasing on $[0,1)$.
\end{prop}

These properties are also linked to the scaling $x = y/\sqrt \alpha$. Indeed, if we 
define $\tilde v_\alpha$ by $\tilde v_\alpha(y) = v_\alpha(y/\sqrt \alpha)$, $\tilde 
v_\alpha$ satisfies 
$-\tilde v_\alpha'' + y^2 \tilde v_\alpha = \frac{\lambda_\alpha}\alpha \tilde 
v_\alpha$, a fact we will use 
extensively in all the proofs in this section. As an example of this scaling, we can 
already prove the following lemma, which was used to get a lower bound on the 
left hand side of the observability inequality in Proposition \ref{th:obs_hol}:
\begin{lemma}\label{th:lemma2}
If we normalize $v_n$ by $v_n(0) = 1$ instead of $|v_n|_{L^2(-1,1)} = 1$, there 
exists $c>0$ such that for all $n\geq 1$,  
$|v_n|_{L^2(-1,1)} \geq c n^{-1/4}$.
\end{lemma}

\begin{proof}
Let us note $\tilde v_n(y) = v_n(y/\sqrt n)$, which is the solution of the Cauchy 
problem $-\tilde v_n'' + y^2 \tilde v_n = \frac{\lambda_n}n\, \tilde v_n$, $\tilde 
v_n(0) = 1$, $\tilde v_n'(0) = 0$. Moreover, $\tilde v_n(\pm \sqrt n) = 0$. Since 
$\lambda_n \sim n$, $\tilde v_n$ converges to the 
solution $\tilde v$ of $-\tilde v'' + y^2 \tilde v = \tilde v$, $\tilde v(0) = 1, 
\tilde v'(0) = 0$, that is $\tilde v(y) = e^{-y^2/2}$, this convergence being uniform 
on every compact subsets of $\set R$.

So, $\int_{-1}^1 \tilde v_n(y)^2 \diff y \xrightarrow[n\to +\infty]{} \int_{-1}^1 
e^{-y^2}\diff y$, and we have $c:= \inf_n \int_{-1}^1 
\tilde v_n(y)^2\diff y >0$. By the change of variables $x = y/\sqrt n$, we have:
\[\int_{-1/\sqrt n}^{1/\sqrt n} v_n(x)^2 \diff x = \frac1{\sqrt n} \int_{-1}^1 
\tilde v_n(y)^2 \diff y \geq \frac c{\sqrt n}\]
and since $\int_{-1}^1 v_n(x)^2 \diff x\geq 
\int_{-1/\sqrt n}^{1/\sqrt n} v_n(x)^2\diff x$, this proves the lemma.
\end{proof}

\subsection{Exponential estimate of the first eigenvalue}
In this subsection, we still normalize $v_\alpha$ so that $v_\alpha(0) = 1$ instead 
of normalizing it in $L^2(-1,1)$. The main result of this section is about refining 
the estimates $\lambda_n\sim n$:
\begin{theorem}\label{th:asymptotic}
  There exists a non-decreasing function $r:(0,\pi/2)\to \set R_+$ and a function
$\gamma$ in $S(r)$ (see Definition \ref{def:symbols}) such that for all reals 
$\alpha>r(0)$, $\lambda_\alpha = \alpha + 
\gamma(\alpha)e^{-\alpha}$.
\end{theorem}

\begin{remark}
\begin{itemize}
 \item This is a semi-classical problem with $h = \frac1\alpha$. The asymptotic 
expansion 
of $\lambda_\alpha$ was already known for $\alpha$ 
real (see for instance \cite{nier_quantitative_2004}), but the estimate in our 
result is also valid for $\alpha$ complex, which was not known before (as far as the 
author knows).
 \item  We we will also prove that for all $\theta$ in $(0,\pi/2)$: 
 \[\gamma(\alpha) \underset{\substack{|\alpha|\to\infty\\ \alpha\in 
U_{\theta,r(\theta)}}}{\sim} 
4\pi^{-1/2}\alpha^{3/2}.\]
A careful examination of the proof even shows that we have an 
asymptotic expansion of the form $\gamma(\alpha) = \sum_{k\geq 0} a_k 
\alpha^{3/2-k}$, this 
expansion being valid in each $U_{\theta,r(\theta)}$, and where the $a_k$ can be in 
principle computed explicitly.
\end{itemize}

\end{remark}

\begin{proof}
  The proof is in three steps. We first explicitly solve the equation satisfied 
by $v_\alpha$ for $\alpha>0$, expressing the solution as an integral on some 
complex path. Then, writing the boundary condition for this explicit solution 
constitute an implicit equation satisfied by $\alpha$ and $\rho_\alpha = 
\lambda_\alpha - \alpha$, this equation still making sense if $\alpha$ is 
complex with positive real part. We use Newton's method to solve this implicit 
equation, with the stationary phase theorem providing the necessary estimates 
for Newton's method to converge. Finally, the stationary phase theorem also 
implies an 
equivalent of the solution Newton's method gives us, which will allow us to 
conclude.
  
  \paragraph{Explicit solution of the equation satisfied by $v_\alpha$}
  Let us recall that $v_\alpha$ satisfies $-v_\alpha''+(\alpha x)^2 v_\alpha = 
\lambda_\alpha v_\alpha$. We have by choice of normalization $v_\alpha(0) = 1$, 
and since $v_\alpha$ is even $v_\alpha'(0) = 0$. Let $w_\alpha$ be defined by 
$v_\alpha = e^{-\alpha x^2/2} w_\alpha$. By developing the derivatives, we have: 
$-w_\alpha'' + 2\alpha x w_\alpha' = (\lambda_\alpha - \alpha)w_\alpha$. 
Finally, we make the change of variables $x = y/\sqrt{\alpha}$, so that $\tilde 
w_\alpha(y) = w_{\alpha}(y/\sqrt\alpha)$ satisfies $-\tilde w_\alpha'' +2y 
\tilde w_\alpha' = (\frac1\alpha\lambda_\alpha-1)\tilde w_\alpha$ as well as 
$\tilde w_\alpha(0) = 1$, $\tilde w_\alpha'(0) = 0$, $\tilde w_\alpha(\sqrt 
\alpha) = 0$. So, for all real $\tilde \rho$, we consider the ordinary 
differential equation:
  \begin{equation}\label{eq:w}
  \begin{aligned}
    -\tilde w'' + 2x \tilde w' - \tilde \rho \tilde w = 0\\
    \tilde w(0) = 1, \tilde w'(0) = 0.
  \end{aligned}
  \end{equation}
  
  Let $g(z) = e^{- z^2/4 -(1+\tilde\rho/2)\ln(z)}$ (with the logarithm chosen 
so that $\ln(1) = 0$ and $\ln$ is continuous on the path we will integrate $g$ 
on). This function satisfies $-z^2 g -2(zg)' -\tilde \rho g = 0$ on any simply 
connected domain of $\set C^\star$. Let $\Gamma_+$ and $\Gamma_-$ be paths in 
$\set C^\star$ from $-\infty$ to $\infty$ going above and below $0$ 
respectively. For instance, we can take $\Gamma_+ = 
(-\infty,-\epsilon]\cup\{\epsilon e^{i(\pi-\theta)},0\leq \theta\leq \pi\}\cup 
[\epsilon,+\infty)$ and $\Gamma_- =(-\infty,-\epsilon]\cup\{\epsilon 
e^{i\theta},-\pi\leq \theta\leq 0\}\cup [\epsilon,+\infty)$ for some 
$\epsilon>0$. Then, by integration by parts, the functions $\tilde w_+$ and 
$\tilde w_-$ defined by $\tilde w_\pm(y) = \int_{\Gamma_\pm} g(z)e^{-yz}\diff z$ 
are solutions of the equation $-\tilde w'' + 2y\tilde w' = \tilde\rho \tilde w$. 
When $\rho <2$, these solutions satisfies:
\begin{align*}
\tilde w_\pm '(0) &= -\int_{\Gamma_\pm} ze^{z^2/4-(1+\tilde\rho/2)\ln(z)}\diff 
z\\
&= - (1+e^{\mp i\frac\pi2\tilde\rho})\int_0^{+\infty}e^{-x^2/4-\tilde \rho/2 
\ln(x)}\diff x.
\end{align*}
Finally, when $\rho<2$, the solution of the equation \eqref{eq:w} 
is up to a constant:
  \begin{equation*}
  \tilde w(y) = 
    \left(
     \big(1+e^{i\frac\pi2 \tilde\rho}\big)\int_{\Gamma_+} -
     \big(1+e^{-i\frac\pi2 \tilde\rho}\big)\int_{\Gamma_-}\right)
     \exp\left(
      -yz -\frac{z^2}4-
      \left(1+\frac{\tilde\rho}2\right)\ln z\right)
    \diff z,
  \end{equation*}
  where we have defined $(a_+\int_{\Gamma_+}+a_-\int_{\Gamma_-})f(s) \diff s = a_+ \int_{\Gamma_+}f(s) \diff s + a_-\int_{\Gamma_-}f(s) \diff s$.
  
  \paragraph{Implicit equation and Newton's method}
  In the case where $\tilde\rho = \tilde\rho(\alpha) := \frac1\alpha\lambda_\alpha 
-1$, the above solution is up to a constant $\tilde w_\alpha$, so $\alpha$ and 
$\tilde \rho(\alpha)$ satisfy $\tilde w(\sqrt\alpha) = 0$ when $\tilde\rho = 
\tilde\rho(\alpha)$. So, let us specify in the above solution $y = \sqrt\alpha$ and 
make the change of variables/change of integration path $z = \sqrt\alpha s$, and 
write $-yz - z^2/4 = -\alpha(1+s/2)^2 + \alpha$:
  \begin{equation*}
  \begin{aligned}
    \tilde w(\sqrt \alpha) = \sqrt{\alpha}e^\alpha\left(
    \big(1+e^{i\frac\pi2 \tilde\rho}\big)\int_{\Gamma_+} -
    \big(1+e^{-i\frac\pi2\tilde\rho}\big)\int_{\Gamma_-}\right)\mkern100mu\\
      \exp\left(
      -\alpha\left(1+\frac s2\right)^2 -
      \left(1+\frac{\tilde\rho}2\right)\ln s\right)
      \diff s.
  \end{aligned}
  \end{equation*}

  So, letting $\Phi(\rho,\alpha)$ be defined by:
  \begin{equation*}
  \begin{aligned}
    \Phi(\rho,\alpha) = \left(
    \big(1+e^{i\frac\pi2 \rho}\big)\int_{\Gamma_+} -
    \big(1+e^{-i\frac\pi2\rho}\big)\int_{\Gamma_-}\right)\mkern100mu\\
      \exp\left(
      -\alpha\left(1+\frac s2\right)^2 -
      \left(1+\frac{\rho}2\right)\ln s\right)
      \diff s,
  \end{aligned}
  \end{equation*}
  and assuming $\tilde \rho(\alpha) <2$, we have $\Phi(\tilde\rho(\alpha),\alpha) = 
0$. When $|\rho| <2$, we even have the equivalence between 
$\Phi(\rho,\alpha) = 0$ and
$\alpha(1+\rho)$ being an eigenvalue of $-\partial_x^2 + (\alpha x)^2$ on 
$(-1,1)$.

Note that the equation $\Phi(\rho,\alpha) 
= 0$ still makes sense if we take $\alpha$ with positive real part, and, as 
stated previously, we want to solve it with Newton's method (Theorem 
\ref{th:newton}). In order to prove the convergence of Newton's method on 
suitable sets (i.e. for each $\theta\in(0,\pi/2)$, a set 
$U_{\theta,r(\theta)}$), we need to estimate $(\partial_\rho \Phi)^{-1}$, 
$\partial_\rho^2 \Phi$, and $\Phi(0,\alpha)$; in particular, we will show that 
the later decays faster than the two former as $|\alpha|$ tends to $+\infty$.

  By differentiating under the integral we have:
  \begin{equation*}
  \begin{split}
    \partial_\rho\Phi(\rho,\alpha) = 
    i\frac\pi2 \left(e^{i\frac\pi2 \rho}\int_{\Gamma_+} + e^{-i\frac\pi2\rho}\int_{\Gamma_-}\right)
      \exp\left(
      -\alpha\left(1+\frac s2\right)^2 -
      \left(1+\frac{\rho}2\right)\ln s\right)
      \diff  s\\
    -\frac12
      \left((1+e^{i\frac\pi2\rho})\int_{\Gamma_+} -
      (1+e^{-i\frac\pi2\rho})\int_{\Gamma_-}\right)
      \exp\left(
      -\alpha\left(1+\frac s2\right)^2 -
      \left(1+\frac{\rho}2\right)\ln s\right)\ln s
      \diff  s\\
  \end{split}
  \end{equation*}
  so, by the stationary phase theorem, with the only critical point being $-2$ 
(see Proposition \ref{th:stationary_phase}):
  \begin{align}
      \partial_\rho\Phi(\rho,\alpha) &=\notag
      \sqrt{\frac\pi\alpha}\bigg( i\frac\pi2
       \left( e^{i\frac\pi2\rho}e^{-(1+\rho/2)(\ln2 + i\pi)}
        + e^{-i\frac\pi2\rho}e^{-(1+\rho/2)(\ln2-i\pi)}\right)\\*
      &\notag\qquad -\frac12 \big(
       (1+e^{i\frac\pi2\rho})e^{-(1+\rho/2)(\ln2 + i\pi)}(\ln2+i\pi)\\*
      &\notag\qquad - 
(1+e^{-i\frac\pi2\rho})e^{-(1+\rho/2)(\ln2-i\pi)}(\ln2-i\pi)\big)\bigg)\\*
      &\notag\qquad + \mathcal O_{\alpha\in 
U_{\theta,1}}(|\alpha|^{-3/2})\\
      \label{eq:partial_phi}\begin{split}&= i\sqrt{\frac\pi\alpha} 2^{-(1+\frac\rho2)}\left(
       \pi \cos\left(\frac{\pi\rho}2\right) - 
       \ln(2)\sin\left(\frac{\pi\rho}2\right)
       \right)\\*
      &\qquad+ \mathcal O_{\alpha\in U_{\theta,1}}(|\alpha|^{-3/2}),\end{split}
  \end{align}
  the $\mathcal O$ being uniform in $|\rho|\leq 1$.
  
  If $\rho_{\max}>0$ is chosen so that for all $|\rho|\leq 
\rho_{\max}$, $|\pi\cos(\pi\rho/2) - \ln(2)\sin(\pi\rho/2)|\geq \frac\pi2$, 
there exists $C'_\theta>0$ and $r(\theta)>0$ such that for all 
$|\rho|<\rho_{\max}$ and $\alpha$ in $U_{\theta,r(\theta)}$, 
$|(\partial_{\rho}\Phi(\rho,\alpha))^{-1}| < C'_\theta{\sqrt{|\alpha|}}$.
  
  Similarly, $\partial_\rho^2\Phi(\rho,\alpha)$ can be expressed in terms of $\int_{\Gamma_\pm} \exp(-\alpha(1+s/2)^2 - {(1+\rho/2)\ln(s)}) \ln(s)^m \diff s$ with $m\in\{0,1,2\}$, so, by the stationary phase theorem, increasing $r(\theta)$ if need be, there exists $C_\theta>0$ such that for all $|\rho|<\rho_{\max}$ and $\alpha$ in $U_{\theta,r(\theta)}$, $|\partial_{\rho}^2\Phi(\rho,\alpha)| < C_\theta\frac1{\sqrt{|\alpha|}}$.

  Now, by explicitly writing the integrals defining $\Phi$, we have for all $\alpha$ with $\Re(\alpha)>0$:
  \begin{align*}
    \Phi(0,\alpha) &= 
     2\int_{\Gamma_+} \exp\Big( 
     -\alpha\big(1+\frac s2\big)^2\Big)\frac1s\diff s -
     2\int_{\Gamma_-} \exp\Big( 
     -\alpha\big(1+\frac s2\big)^2\Big)\frac1s\diff s\\
     &=2\lim_{\epsilon\to 0} 
     \int_\pi^{-\pi} \exp\Big(
     -\alpha\big(1+\frac12\epsilon e^{i\theta}\big)^2\Big)
     i\diff \theta\\
     &= -4i\pi e^{-\alpha},
  \end{align*}
  so, increasing again $r(\theta)$ if necessary, we have for all $\alpha$ in $U_{\theta,r(\theta)}$:
 
  \[|\Phi(0,\alpha)|\leq \min((2C_\theta 
{C'_\theta}^2)^{-1},\frac15{C'_\theta}^{-1})|\alpha|^{-1/2}.\]

  Then according to Theorem \ref{th:newton}, with $R = \rho_{\max}/10$, $C_1 = 
C_\theta|\alpha|^{-1/2}, C_2=C'_\theta|\alpha|^{1/2}$ and with starting point $z_0 = 
0$, the sequence $(\tilde \rho_n(\alpha))$ defined by $\tilde\rho_0(\alpha) = 0$, 
$\tilde\rho_{n+1}(\alpha) = \tilde\rho_n(\alpha) - 
\partial_\rho\Phi(\tilde\rho_n(\alpha),\alpha)^{-1}\Phi(\tilde\rho_n(\alpha),\alpha)$ 
converges and the limit $\tilde\rho_\infty(\alpha)$ satisfies 
$|\tilde\rho_\infty(\alpha) - \tilde\rho_k(\alpha)|\leq 
C|A\sqrt{\alpha}e^{-\alpha}|^{2^k}$ for some $C>0$ and $A>0$.

  \paragraph{Equivalent of the solution and conclusion}
  Let us first prove that $\tilde\rho_\infty$ is holomorphic.  By induction, every 
$\tilde\rho_k$ is holomorphic, and the estimate $|\tilde\rho_\infty(\alpha) - 
\tilde\rho_k(\alpha)|\leq C|A\sqrt\alpha e^{-\alpha}|^{2^k}$ shows that 
$\tilde\rho_k$ converges 
uniformly in $U_{\theta,r(\theta)}$ (provided that $r(\theta)>0$), so 
$\tilde\rho_\infty$ is also holomorphic.

  Now let us compute an equivalent of $\tilde\rho_\infty$. According to the previous 
estimate with $k = 1$, we have $\tilde\rho_\infty(\alpha) = 
-\partial_\rho\Phi(0,\alpha)^{-1}\Phi(0,\alpha) + \mathcal O(e^{-2\alpha})$ for 
$\alpha \in U_{\theta,r(\theta)}$. Thanks to the stationary phase theorem, or more 
specifically equation \eqref{eq:partial_phi}, we have: $\partial_\rho \Phi(0,\alpha) 
= i{\pi^{3/2}}\alpha^{-1/2} + \mathcal O_{\alpha\in 
U_{\theta,r(\theta)}}(|\alpha|^{-3/2})$, and $\Phi(0,\alpha) = -4i\pi e^{-\alpha}$. 
So, we have: $\tilde\rho_\infty(\alpha) = 4\pi^{-1/2} \alpha^{1/2}e^{-\alpha}(1 + 
\mathcal O(|\alpha|^{-1}))+ \mathcal O(e^{-2\alpha})$, and since $e^{-2\alpha} = 
\mathcal O(|\alpha|^{-1/2}e^{-\alpha})$ for $\alpha\in U_{\theta,r(\theta)}$, we 
finally have $\tilde\rho_\infty(\alpha) \sim 4\pi^{-1/2}\alpha^{1/2}e^{-\alpha}$ for 
$\alpha \in U_{\theta,r(\theta)}$.

  We still have to check that for $\alpha$ real, $\tilde\rho_\infty(\alpha)$ is 
equal to $\tilde \rho_\alpha$ (let us remind that $\lambda_\alpha = 
\alpha(1+\tilde\rho_\alpha)$). According to equation \eqref{eq:partial_phi}, we 
have for some $C''_\theta>0$ and for all $|\rho|\leq \rho_{\max}$ and $\alpha 
\in U_{\theta,r(\theta)}$: $\Im(\partial_\rho \phi(\rho,\alpha))\geq 
C''_\theta/\sqrt{|\alpha|}$. So for $|\rho|<\rho_{\max}$ and $\alpha\in 
U_{\theta,r(\theta)}$, $|\Phi(\rho,\alpha)| = |\Phi(\rho,\alpha) - 
\Phi(\tilde\rho_\infty(\alpha))| \geq C''_\theta |\rho - 
\tilde\rho_\infty(\alpha)|/\sqrt{|\alpha|}$. So for $\alpha$ real big enough, 
$\tilde\rho_\infty(\alpha)$ is the smallest non-negative zero of 
$\Phi(\cdot,\alpha)$.
  
  So $\tilde\rho_\infty$ is the smallest positive eigenvalue of $-\partial_x^2 
+ (\alpha x)^2$, and since these eigenvalues are all positive (Proposition
\ref{th:basic_eigen}), we actually have $\lambda_\alpha = 
\alpha(\tilde\rho_\infty(\alpha) + 1)$.
\end{proof}

\begin{remark}
 This proof is the one we are not (yet?) able to carry if we replace in the Grushin 
equation \eqref{eq:grushin} the potential $x^2$ by $a(x) = x^2+x^3b(x)$ where $b$ 
is any non null analytic function. Indeed, the proof above relies on an exact 
integral representation of the solution of $-v'' + a(x) v = \lambda v$, which is 
impossible in general if $b\neq  0$.
\end{remark}

\subsection{Agmon estimate for the first eigenfunction}
Thanks to the Theorem \ref{th:asymptotic}, we can define $\lambda_\alpha$ for $\alpha 
\in \bigcup U_{\theta,r(\theta)}$ by $\lambda_\alpha = \alpha + \gamma(\alpha) 
e^{-\alpha}$, and $v_\alpha$ as the solution of $-v_\alpha'' + (\alpha x)^2 v_\alpha 
= \lambda_\alpha v_\alpha$, $v_\alpha(0) = 1, v'(\alpha) = 0$. As a solution of an 
ordinary differential equation that depends analytically on a parameter $\alpha$, 
$v_\alpha(x)$ depends analytically on $\alpha$, and we have thus $v_\alpha(\pm 1) = 
0$. We now prove some estimates on $v_\alpha(x)$, in the form of the following 
proposition:
\begin{prop}\label{th:est_agmon}
Let $1\geq \epsilon>0$ and ${^\epsilon w(x)} (\alpha) = 
e^{\alpha(1-\epsilon)x^2/2}v_\alpha(x)$. There exists a non-decreasing function 
$r:(0,\pi/2)\to \set R_+$ such that 
${^\epsilon w}$ is bounded from $[-1,1]$ to $S(r)$.
\end{prop}

\begin{remark}
  The boundedness of $^\epsilon w$ in the statement of the Proposition 
\ref{th:est_agmon} is to be understood as the boundedness of the subset 
$\{{^\epsilon w(x)},x\in[-1,1]\}$ of $S(r)$, which is equivalent to the fact that for 
all seminorms $p_{\epsilon',\theta}$, the set $\{p_{\epsilon',\theta}({^\epsilon 
w(x)}),x\in[-1,1]\}$ is a bounded set of $\set R$.
\end{remark}
\begin{proof}
  This is mostly a complicated way of stating Agmon's estimate (see for instance 
Agmon's initial work \cite{agmon_lectures_1982} or Helffer and Sjöstrand's 
article \cite{helffer_multiple_1984}, the later being closer to what we are 
doing).
  
  Let $\theta_0\in(0,\pi/2)$. We will prove that there exists $C>0$ and 
$r'(\theta_0)$ such that for all $\alpha \in U_{\theta_0,r'(\theta_0)}$ and 
$x\in (-1,1)$, ${^\epsilon w(\alpha)(x)} \leq C |\alpha|^{3/4}$, which is enough 
to prove the stated proposition. In this proof, we will just note $w$ instead of 
${^\epsilon w}$, and for convenience, we will note $w_\alpha(x)$ instead of 
$w(x)(\alpha)$.
  
  For all $\alpha$, we have: $-w_\alpha'' +2\alpha{(1-\epsilon)}x 
w_\alpha'+\big({(1-(1-\epsilon)^2)}(\alpha x)^2 - 
\epsilon\alpha-\rho(\alpha)\big) w_\alpha = 0$. Let us write $\alpha = \frac1h 
e^{i\theta}$, $\delta^2 = 1-(1-\epsilon)^2$, and multiply the previous equation 
by $h^2e^{-i\theta}\bar w_\alpha$. We get: 
\[-h^2e^{-i\theta}w_\alpha''\bar w_\alpha +2h(1-\epsilon)xw_\alpha' \bar w_\alpha + 
\big(e^{i\theta}\delta^2 x^2 - h(\epsilon + 
he^{-i\theta}\rho(\alpha))\big)|w_\alpha(x)|^2 = 0.\]

By integration 
by parts, we have $-\int_{-1}^1 w_\alpha''(x) \bar w_\alpha(x) \diff x = 
\int_{-1}^1|w_\alpha'(x)|^2 \diff x$ and since $2\Re(w_\alpha'\bar w_\alpha) = 
\frac{\diff}{\diff x}|w_\alpha|^2$, we have $2\!\int_{-1}^1 x 
\Re(w_\alpha'(x)\bar 
w_\alpha(x))\diff x = -\int_{-1}^1 |w_\alpha(x)|^2 \diff x$, so integrating the 
equation and taking the real part, we get the Agmon estimate, valid for all 
$\alpha$ such that $\rho(\alpha)$ is defined:
  \begin{equation}\label{eq:agmon_est}
    h^2\int_{-1}^1|w_\alpha'(x)|^2\diff x + \int_{-1}^1 \left( \delta^2 x^2 
-h\frac{1+h\Re(e^{-i\theta}\rho(\alpha))}{\cos(\theta)}\right)|w_\alpha(x)|^2 \diff x 
= 0.
  \end{equation}
  
  The final ingredient we need to conclude is a comparison between $w_\alpha$ and $e^{-\epsilon \alpha x^2/2}$, which will give us a control of the $L^2$ norm of $w_\alpha$ on sets of the form $(-R|\alpha|^{-1/2},R|\alpha|^{-1/2})$. The function $\tilde w$ defined by $\tilde w(z) = e^{-\epsilon z^2/2}$ satisfies $-\tilde w'' + 2(1-\epsilon) z \tilde w' +\delta^2z^2\tilde w - \epsilon \tilde w = 0$ for $z$ in $\set C$, so the solution $\tilde w_\rho$ of $-\tilde w_\rho'' +2(1-\epsilon)z \tilde w_\rho' + \delta^2 z^2 \tilde w_\rho = (\epsilon+\rho)\tilde w_\rho$ tends to $e^{-\epsilon z^2/2}$ in $L^2(D(0,R))$ as $\rho$ tends to $0$. So, for all $R>0$, there exists $\rho_{\max}$ such that for $|\rho|\leq \rho_{\max}$, $|\tilde w_\rho - e^{-\epsilon z^2/2}|_{L^2(D(0,R))} \leq 1$. But $w_\alpha(x) = \tilde w_{\rho(\alpha)/\alpha}(\sqrt\alpha x)$, so, if $\rho(\alpha)/\alpha \leq \rho_{\max}$, we have $|w_\alpha - e^{-\epsilon\alpha x^2/2}|_{L^2(|x|\leq R/\sqrt{|\alpha|})} \leq |\alpha|^{-1/2}$. So, there exists $r'(\theta_0)\geq r(\theta_0)$ such that for all $\alpha$ in $U_{\theta_0,r'(\theta_0)}$:
  \begin{equation}\label{eq:comp_eigenvectors}
  |w_\alpha|_{L^2(-R/\sqrt{|\alpha|},R/\sqrt{|\alpha|})}\leq C_{\epsilon,\theta_0} 
|\alpha|^{-1/4}.
  \end{equation}
  
  Let $E = \{x\in (-1,1), \delta^2 x^2 -2h/\cos(\theta_0)\leq 0\} = \{|x|\leq \sqrt{2}/(\delta \sqrt{\cos(\theta_0)}) |\alpha|^{-1/2}\}$ and $\alpha = \frac1h e^{i\theta}$ in $U_{\theta_0,r'(\theta_0)}$. We have $|h\Re(e^{-i\theta}\rho(\alpha))|\leq 1$ and $|\theta|<\theta_0$, so, for $x$ in $[-1,1]\setminus E$, $\delta^2 x^2 -h(1+h\Re(e^{-i\theta}\rho(\alpha)))/\cos(\theta)>0$. Thus, thanks to Agmon estimate \eqref{eq:agmon_est}, $h^2|w_\alpha'|^2_{L^2(-1,1)} \leq C'_{\theta_0}|w_\alpha|^2_{L^2(E)}$. But, thanks to inequality \eqref{eq:comp_eigenvectors}, we have $|w_\alpha|_{L^2(E)}^2 \leq C_{\epsilon,\theta_0} h^{1/2}$, so $|w_\alpha'|_{L^2(-1,1)}^2 \leq C'_{\epsilon,\theta_0} h^{-3/2}$.
  
  Finally, for all $x$ in $(-1,1)$, we have: ${|w_\alpha(x) - w_\alpha(0)|}\leq |w_\alpha'|_{L^1(-1,1)}$, so thanks to Hölder's inequality, $|w_\alpha(x) - w_\alpha(0)| \leq \sqrt2|w_\alpha '|_{L^2(-1,1)}\leq \sqrt{2C'_{\epsilon,\theta_0}} h^{-3/4}$.
\end{proof}

\subsection{Proof of Lemma \ref{th:lemma1}}\label{sec:proof_lemma}
We prove here Lemma \ref{th:lemma1}. To bound from above $\left|\sum v_n(x) a_n 
z^{n-1} |\zeta|^{\rho_n}\right|$, the idea is to apply Theorem \ref{th:est_default}, 
with Theorems \ref{th:asymptotic} and \ref{th:est_agmon} providing the required 
hypotheses.

First, in order to apply Theorem \ref{th:est_default}, we define some symbols.
Let $\gamma\in S(r_1)$ obtained by Theorem \ref{th:asymptotic}, and 
$v:(-1,1)\to S(r_2)$ the function obtained by Proposition \ref{th:est_agmon} 
with $\epsilon = 1$. By taking $r = \max(r_1,r_2)$, we can assume that 
$\gamma\in S(r)$ and that $v$ take its values in $S(r)$. This $v$ is 
still bounded (see Proposition \ref{th:change_domain}, first item). Finally, for 
$\zeta\in \mathcal D$ and $x\in(-1,1)$, let $\gamma_{\zeta,x}$ defined 
by\footnote{We remind that $\mathcal D = \{e^{-T}<|z|<1, \arg(z) \in \omega_y\}$, 
see 
Section \ref{sec:toy_model} and figure \ref{fig:model}. Also, $U = \{0<|z|<1, 
\arg(z)\in \omega_y\}$ and $U^\delta = \{z, \distance(z,U)\in \omega_y\}$.}:
\[\gamma_{\zeta,x}(\alpha) = v(x)(\alpha)|\zeta|^{\rho(\alpha)},\]
so that:
\begin{equation}\label{eq:lemma}\sum_{n> r(0)} v_n(x) a_n z^{n-1}|\zeta|^{\rho_n} = 
\frac1zH_{\gamma_{\zeta,x}}\left(\sum_{n>r(0)} a_n z^{n}\right).
\end{equation}

We then show that the family $(\gamma_{\zeta,x})_{\zeta\in \mathcal D, x\in(-1,1)}$ 
is in $S(r)$, and is bounded. We already know that $(v(x))_{x\in(-1,1)}$ 
is a bounded family in $S(r)$. Since the multiplication is continuous in $S(r)$ 
(Proposition \ref{th:prop_symbols}), to prove that $(\gamma_{\zeta,x})$ is a bounded 
family, it is enough to prove that $(|\zeta|^{\rho})_{\zeta\in \mathcal D}$ is a
bounded family of $S(r)$.

Since $\rho(\alpha) = e^{-\alpha} \gamma(\alpha)$ with $\gamma$ having 
sub-exponential growth (by definition of $S(r)$), $|\rho(\alpha)|$ is bounded on 
every 
$U_{\theta,r(\theta)}$ by some $c_\theta$. So, we have for $\zeta\in \mathcal D$ and 
$\alpha\in U_{\theta,r(\theta)}$:
\begin{equation*}
 ||\zeta|^{\rho(\alpha)}| \leq e^{-\ln|\zeta|c_\theta} \leq e^{T c_\theta}.
\end{equation*}

So $|\zeta|^{\rho(\alpha)}$ is bounded for $\alpha \in U_{\theta,r(\theta)}$, and 
in particular has sub-exponential growth. Since $\rho$ is holomorphic, so is 
$\alpha\mapsto |\zeta|^{\rho(\alpha)}$, thus, $\alpha\mapsto |\zeta|^{\rho(\alpha)}$ 
is in $S(r)$. Moreover, the 
bound $||\zeta|^{\rho(\alpha)}|\leq e^{Tc_\theta}$ is uniform in $\zeta\in 
\mathcal 
D$, so $(|\zeta|^{\rho})_{\zeta\in \mathcal D}$ is a bounded family of 
$S(r)$.

\pushQED{\qed}
We have proved $(\gamma_{\zeta,x})$ is a bounded family of $S(r)$, so according 
to the estimate on holomorphy default operators (Theorem \ref{th:est_default}), if 
$V$ is a bounded domain that is star-shaped with respect to $0$, for any 
$\delta'>0$, there exists 
$C>0$ independent of $\zeta,x$, such that:
\[\bigg|\sum_{n> r(0)} \gamma_{\zeta,x}(n) a_n z^n\bigg|_{L^\infty(V)}\mkern-10mu 
\leq 
C\bigg|\sum_{n> r(0)} a_nz^n\bigg|_{L^\infty(V^{\delta'})}.\]

We can't apply this estimate directly with $U=V$ since $0\notin U$, but we can 
choose $V$ and $\delta'$ such that $U\subset V$ and $V^{\delta'}\subset 
U^\delta$ (for instance $\delta' = \delta/2$ and $V = U^{\delta'}$): there exists 
$C>0$ independent of $\zeta\in \mathcal D$ and $x\in(-1,1)$: \[\bigg|\sum_{n> r(0)} 
\gamma_{\zeta,x}(n) a_n z^n\bigg|_{L^\infty(U)}  \mkern-10mu\leq
C\bigg|\sum_{n> r(0)} a_nz^n\bigg|_{L^\infty(U^{\delta})}.\]
So, thanks to equation \eqref{eq:lemma}:
\begin{equation*}
\bigg|\sum_{n> r(0)} v_n(x) a_n z^{n-1}|\zeta|^{\rho_n}\bigg|_{L^\infty(\mathcal 
D)}\mkern-15mu
\leq  Ce^T\bigg|\sum_{n> r(0)} a_nz^n\bigg|_{L^\infty(U^{\delta})}\mkern-15mu
\leq  Ce^T\bigg|\sum_{n> r(0)} 
a_nz^{n-1}\bigg|_{L^\infty(U^{\delta})}\mkern-5mu.\qedhere
\end{equation*}
\popQED

\section{Conclusion and open problems}
We proved the non-null controllability of the Grushin equation on some special 
control domain, and  if we 
combine our result with the previous ones \cite{beauchard_null_2014-1, 
beauchard_2d_2015}, all of the following situations can happen depending on the 
control domain $\omega$:
\begin{itemize}
  \item the Grushin equation is controllable in any time, for instance if $\omega = (0,a)\times(0,1)$;
  \item the Grushin equation is controllable in large time, but not in small time, for instance if $\omega = (a,b)\times(0,1)$ with $0<a<b$;
  \item the Grushin equation is never controllable, for instance if $\omega = (-1,1)\times((0,1)\setminus[a,b])$ with $0\leq a < b \leq 1$.
\end{itemize}

A pattern that seems to appear in the controllability of degenerate parabolic equations is that the controllability holds in any time when the degeneracy is weak, and never holds when the degeneracy is strong. Our result may indicate that obtaining general results on the controllability of parabolic equation degenerating inside the domain will be difficult in the critical case, i.e. when the degeneracy is neither strong nor weak.

On the Grushin equation, null-controllability is still an open problem for domains 
that do not fall into one of the three domain type we described before. Also, 
controllability of higher-dimension Grushin equation, for $x \in (-1,1)$, $y \in 
\set T^n$, on $\omega = (-1,1) \times \omega_y$ is still an open problem (the case 
where $\omega = (a,b)\times \set T^n$ is mentioned in \cite{beauchard_null_2014-1}).

Another question we might ask is whether regular initial conditions can be steered 
to $0$, as it happens for the Grushin equation when the control domain is two 
symmetric vertical bands \cite{beauchard_2d_2015}. The answer 
is negative:
\begin{prop}\label{th:main_aug}
 Let $T>0$ and $\omega$ as in the main theorem (Theorem \ref{th:main}). For 
$\alpha>0$, let $\mathcal A_\alpha = \{\sum a_n(x) e^{iny}, \sum |a_n|_{L^2(-1,1)}^2 
e^{2\alpha n} <+\infty\}$. Then, for every $\alpha>0$, there exists an initial 
condition $f_\alpha$ in $\mathcal A_\alpha$ that cannot be steered in time $T$ to $0$ 
by means of $L^2$ controls localized in $\omega$.
\end{prop}

\begin{proof}
 According to Theorem \ref{th:main}, there exists an initial condition $f_0 = \sum 
a_n v_n(x) e^{iny}$ in $L^2(\Omega)$ that cannot be steered to $0$ by a $L^2$ 
control 
localized on $\omega$ in time $T+\alpha$. Let $f(t,x,y)$ be the solution of the 
Grushin equation \eqref{eq:grushin} with $f_0$ as the initial condition, and let 
$f_\alpha(x,y) = f(\alpha,x,y)$. Then, since $\lambda_n = n + o(1)$, $f_\alpha(x,y) = 
\sum a_n v_n(x) e^{-\alpha\lambda_n}e^{iny}$ is in $\mathcal A_\alpha$, and if it 
could be steered to $0$ in time $T$, then, $f_0$ could be steered to $0$ in time 
$T+\alpha$.
\end{proof}

From this proposition, we could ask if there is even one non-null initial condition 
that can be steered to $0$. For the moment, it is unknown, but we conjecture that 
there 
is none.

\begin{appendices}
\section{The stationary phase theorem}
We prove here the following theorem:
\begin{theorem}
  There exists $C>0$ such that for all $u\in \mathcal S(\set R)$, $N\in \set N$ and $\alpha\in \{\alpha\neq0,\Re(\alpha)\geq 0\}$:
  \[
    \int_{\set R} e^{-\alpha x^2/2} u(x) \diff x =
    \sum_{k=0}^{N-1} \frac{\sqrt{2\pi}}{2^k k! \alpha^{k+1/2}} u^{(2k)}(0) + S_N(u,\alpha)
  \]
  where $\alpha^s$ is defined to be $e^{s\ln \alpha}$, with the principal determination of the logarithm, and $S_N(u,\alpha)$ satisfying:
  \[
    |S_N(u,\alpha)| \leq \frac{C}{2^N N! |\alpha|^{N+1/2}}\sum_{k=0}^2 
\|u^{(2N+k)}\|_{L^1}.
  \]
\end{theorem}

\begin{proof}
  Since the proof is essentially the same as the one provided by Martinez 
\cite[][theorem 2.6.1]{martinez_introduction_2002} for the case $\alpha$ purely 
imaginary, we just give the main ideas.
  
  We define the Fourier transform of $u$ in the Schwartz space by $\mathcal 
F(u)(\xi) = \int_{\set R} u(x)e^{-ix\xi}\diff x$. Then, the Fourier transform of 
$x\mapsto e^{-\alpha x^2/2}$ is $\mathcal F(e^{-\alpha x^2/2})(\xi) = 
\sqrt{\frac{2\pi}{\alpha}}e^{-\xi^2/{2\alpha}}$ (this is standard when 
$\Re(\alpha)>0$, and by taking the limit in $\mathcal S'(\set R)$ for $\alpha+ 
\epsilon\to\alpha$ when $\alpha$ is purely imaginary). So, we have:
  \[\int_{\set R}e^{-\alpha x^2/2}u(x)\diff x = \frac{1}{\sqrt{2\pi\alpha}}\int_{\set 
R}e^{-\xi^2/(2\alpha)}\mathcal F(u)(\xi)\diff \xi.\]
  
  Then, writing
  \[e^{-\xi^2/(2\alpha)} = \sum_{k\geq 0}\frac{(-1)^k}{(2\alpha)^k k!}\xi^{2k} = 
\sum_{k< N}\frac{(-1)^k}{(2\alpha)^k k!}\xi^{2k} + R_N(\xi,\alpha),\]
  with, according to Taylor's formula,
  \[|R_N(\xi,\alpha)|\leq\frac{\xi^{2N}}{2^N |\alpha|^N N!},\]
  we have
  \[\int_{\set R} e^{-\alpha x^2/2}u(x)\diff x = \sum_{k<N} \frac{\sqrt{2\pi}}{2^k k! \alpha^{k+1/2}}u^{(2k)}(0) + S_N(u,\alpha)\]
  with $S_N(u,\alpha) = \int_{\set R} R_N(\xi,\alpha)\mathcal F u(\xi)\diff \xi$, that satisfies the estimate stated in the theorem.
  
  We refer to Martinez's proof in the already mentioned book 
\cite{martinez_introduction_2002} for more details on the computations.
\end{proof}

Here is the particular case of the stationary phase theorem we will need:
\begin{prop}\label{th:stationary_phase}
  Let $\Gamma_{+,\epsilon}$ be the path $(-\infty,\epsilon]\cup \{\epsilon e^{i(\pi-\theta)},0\leq \theta \leq \pi\}\cup [\epsilon,+\infty)$. Let $\theta$ in $(0,\pi/2)$. There exists $C_\theta>0$ such that for all $\alpha\in U_{\theta,1}$, $\rho \in D(0,1)$ and $m\in\{0,1,2\}$, with $f(s) = e^{-(1+\rho/2)\ln(s)}(\ln(s))^m$:
  \[
    \left|\int_{\Gamma_{+,\epsilon}} e^{-\alpha(1+s/2)^2} f(s) \diff s - 
2\sqrt{\frac{\pi}{\alpha}} f(-2)\right| \leq C_\theta|\alpha|^{-3/2}.
  \]
\end{prop}

\begin{proof}
  We start by choosing $\chi\in C_c^\infty(-3,-1)$ with $\chi = 1$ on $(-5/2,-3/2)$, and we modify slightly the path $\Gamma_{+,\epsilon}$ so that it is a $C^\infty$ path (the result of the integral of course not depending on this modification of $\Gamma_{+,\epsilon}$), for instance, if $\phi$ is in $C_c^\infty(-\epsilon,\epsilon)$ with $\phi\geq 0$ and $\phi(0)>0$, we can choose $\Gamma_{+}(t) = t + 2i\phi(t)$. Then we write:
  
 \begin{equation*}
 \begin{aligned}
\int_{\Gamma_{+,\epsilon}} e^{-\alpha(1+s/2)^2} f(s)\diff s &= \int_{-3}^{-1} e^{-\alpha(1+t/2)^2}\chi(t)f(t)\diff s \\
&+\int_{\set R\setminus(-5/2,-3/2)} 
e^{-\alpha(1+\Gamma_+(t)/2)^2}(1-\chi(t))f(\Gamma_+(t))\diff t.
 \end{aligned}
 \end{equation*}
 
 We can apply the previous theorem to the first term, so we only need to show that for some $C'_\theta>0$, the second term is bounded by $C'_\theta|\alpha|^{-3/2}$. We note $\varphi(t) = (1+\Gamma_+(t)/2)^2 = (1+t/2)^2 - \phi^2(t) +2i\phi(t)(1+t/2)$ whose only critical point is $-2$, and let $L$ the operator defined by $L u = \frac1{\varphi'}u'$ so that $L e^{-\alpha\varphi} = -\alpha e^{-\alpha\varphi}$. So, noting $L^t u = (\frac1{\varphi'}u)'$, we have by integration by parts:
 \begin{equation*}
 \begin{split}
\int_{\set R\setminus(-5/2,-3/2)} e^{-\alpha(1+\Gamma_+(t)/2)^2}(1-\chi(t))f(\Gamma_+(t))\diff t =\\
 \frac1{\alpha^2}\int_{\set R \setminus (-5/2,-3/2)} e^{-\alpha \varphi(t)} (L^t)^2 ((1-\chi)f\circ \Gamma_+)(t) \diff t
 \end{split}
 \end{equation*}
 
 Then, writing $|e^{-\alpha\varphi(t)}| \leq e^{-\Re(\alpha\varphi(t))}$ and $\Re(\alpha\varphi(t)) = \Re(\alpha) ((1+t/2)^2 - \phi^2(t)) - 2{(1+t/2)}\Im(\alpha)\varphi(t))$. If $|\arg(\alpha)|<\theta$, then for some $c_\theta>0$, $\Re(\alpha) \geq c_\theta |\alpha|$, and if we choose $\phi$ small enough, we have for some $c'_\theta>0$ and all $t \notin (-5/2,-3/2)$, $\Re(\alpha\varphi(t))\geq c'_\theta |\alpha| (1+t/2)^2$. So, for all $\alpha$ in $U_{\theta,1}$:
\begin{equation*}
 \begin{split}
\left|\int_{\set R\setminus(-5/2,-3/2)} e^{-\alpha(1+\Gamma_+(t)/2)^2}(1-\chi(t))f(\Gamma_+(t))\diff t\right| \leq\\
\frac1{|\alpha|^ 2} |(L^t)^ 2((1-\chi) f\circ \Gamma_+)|_{L^\infty(\set R\setminus(-5/2,-3/2))} \int_{\set R} e^{-c'_\theta|\alpha|t^2/4}\diff t
 \end{split}
 \end{equation*}
 which concludes the proof.
\end{proof}

\section{Newton's method}
We prove here that Newton's method can solve equations of the form $\Phi(z) = 0$ by an iterative scheme, assuming the starting point is is close enough to a solution. While such a theorem can be stated in Banach spaces, we will only need it in the complex plane.
\begin{theorem}\label{th:newton}
  Let $D=D(0,R)$ be a disk in the complex plane. We will note $5D = D(0,5R)$ and $6D = D(0,6R)$. Let $\Phi: 6D\to \set C$ be an holomorphic function such that:
  \begin{itemize}
    \item for all $z \in D$, $\Phi(z) \in D$;
    \item For all $z$ in $6D$, $\Phi'(z) \neq 0$.
  \end{itemize}
  
  Then, noting $C_1 = \sup_{5D}|\Phi''|$ and $C_2 = \sup_{5D}|\Phi'^{-1}|$ and $A = C_1C_2^2$, if $z_0$ is in $D$ and $|\Phi(z_0)|\leq \min((2A)^{-1},2RC_2^{-1})$, then the sequence $(z_n)$ defined by $z_{n+1} = z_n - \Phi'(z_n)^{-1}\Phi(z_n)$ converges and the limit $z_\infty$ satisfies $\Phi(z_\infty) = 0$. Moreover, for all $k\geq 0$, $|z_\infty - z_k| \leq \frac2{C_1C_2} |A\Phi(z_0)|^{2^k}$.
\end{theorem}

\begin{proof}
  Let $\epsilon_0 = |\Phi(z_0)|$. We prove by induction the predicate $P(n): z_n \in 
5D$ and $|\Phi(z_n)|\leq A^{-1} (A \epsilon_0)^{2^n}$. About the case $n=0$, we made 
the hypothesis $z_0\in D$, while the inequality just reads $|\Phi(z_0)|\leq 
|\Phi(z_0)|$.
  
  Now suppose that $P(k)$ holds for all $k\leq n$. Let $v_k = -\Phi'(z_k)^{-1}\Phi(z_k)$, so that $z_{n+1} = z_0 + v_0 + \dotsb v_n$. By definition of $C_1$ and $C_2$ and the fact that for all $k\leq n$, $z_k \in 5D$, $|z_{n+1}|\leq |z_0| + \sum_{k=0}^n C_2 A^{-1}(A\epsilon_0)^{2^k} \leq R + C_2 \epsilon_0\frac1{1-A\epsilon_0}$. Since we have by hypothesis $\epsilon_0\leq \min((2A)^{-1},2RC_2^{-1})$, we have $z_{n+1} \leq R+2C_2\epsilon_0\leq 5R$, which proves that $z_{n+1}\in 5D$.
  
  In order to prove that $|\Phi(z_{n+1})|\leq A^{-1}(A\epsilon_0)^{2^{n+1}}$, we make a Taylor expansion of $\Phi$ about $z_n$: for all $\delta \in \set C$ such that $z_n+\delta\in 5D$:
  \[|\Phi(z_n+\delta) - \Phi(z_n) - \delta\Phi'(z_n)|\leq \frac12 C_1 |\delta|^2.\]
  We then choose $\delta$ so that $\Phi(z_n) + \delta\Phi'(z_n) = 0$. With this 
$\delta$, the previous inequality is $|\Phi(z_{n+1})|\leq \frac12 C_1 
|\Phi(z_n)|^2|\Phi'(z_n)^{-1}|^2$. So, by the definition of $C_2$ and the 
induction hypothesis, $|\Phi(z_{n+1})|\leq \frac12C_1C_2^2 
\big(A^{-1}(A\epsilon_2)^{2^n}\big)^2 = \frac12 A^{-1} (A\epsilon_0)^{2^{n+1}}$, 
which ends the proof of the induction.
  
  By the same kind of computations we made in order to prove $z_{n+1}\in 5D$, we have for $n\geq k$: $|z_n-z_k|\leq C_2 \sum_{j=k}^{n-1}A^{-1}(A\epsilon_0)^{2^j}\leq C_2A^{-1} (A\epsilon_0)^{2^k}\frac1{1-A\epsilon_0}\leq 2C_2 A^{-1}(A\epsilon_0)^{2^j}$. This proves the stated estimate and that $(z_n)$ converges. Since we have $|\Phi(z_n)|\leq A^{-1}(A\epsilon_0)^{2^n}$, the limit $z_\infty$ satisfies $\Phi(z_\infty) = 0$.
\end{proof}

\section{The Grushin equation on the rectangle}\label{sec:grushin_rectangle}
We look here at the Grushin equation $\partial_t f - \partial_x^2 f - 
x^2\partial_y^2f = \mathds 1_\omega$ with $(x,y) \in \Omega = (-1,1)\times (0,1)$ and 
with Dirichlet boundary conditions on $\partial\Omega$. The situation is the 
same as the Grushin equation on the torus $(-1,1)\times \set T$:
\begin{theorem}
  Let $[a,b]$ be a non trivial segment of $(0,1)$, $\omega_y = (0,1)\setminus 
[a,b]$,  $\omega = (-1,1)\times \omega_y$ and $T>0$. The Grushin equation on 
$\Omega$ is not controllable on $\omega$ in time $T$.
\end{theorem}
\begin{proof}[Sketch of the proof]
This time, we look for a counterexample of the observability inequality among 
linear combinations of the eigenfunctions $\Phi_n$ defined by $\Phi_n(x,y) = 
v_{n\pi}(x)\sin(n\pi y)$. Then, writing $\sin(n\pi y) = \frac1{2i}(e^{in\pi y} - 
e^{-in \pi y})$, we have:
\begin{multline*}
\int_{\substack{0<t<T\\-1<x<1\\y\in\omega_y}}\left|\sum a_n 
e^{-\lambda_{n\pi}t}\Phi_n(x,y)\right|^2\diff t\diff x\diff y\\
\leq \frac12\!
\int_{\substack{0<t<T\\-1<x<1\\y\in\omega_y}}\left(
\left|\sum a_n e^{-\lambda_{n\pi}t}v_{n\pi}(x)e^{in\pi y}\right|^2\!\!+
\left|\sum a_n e^{-\lambda_{n\pi}t}v_{n\pi}(x)e^{-in\pi y}\right|^2
\right)\diff[-6] t\diff x\diff y.
\end{multline*}

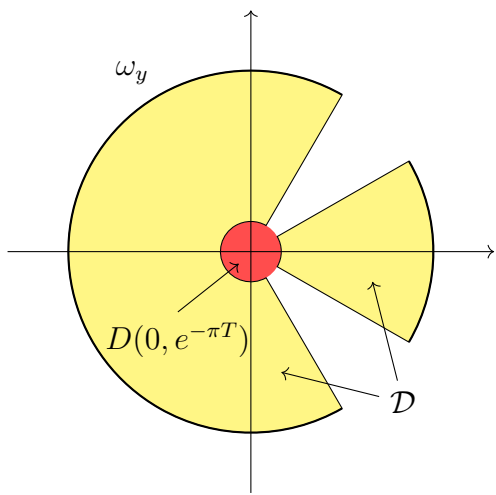
\begin{figure}[ht]
 \begin{minipage}[c]{0.5\textwidth}
  \begin{center}
   \begin{tikzpicture}[scale=0.8]
  \fill[red!70] (0,0) circle[radius=0.5];
  \draw[fill = yellow!60] (60:3)
  arc[radius=3, start angle = 60, delta angle = 240]
  -- (-60:0.5)
  arc[radius=0.5, start angle = -60, delta angle = -240]
  --cycle;
  \draw[fill = yellow!60] (-30:3)
  arc[radius=3, start angle = -30, delta angle = 60]
  -- (30:0.5)
  arc[radius=0.5, start angle = 30, delta angle = -60]
  --cycle;
  \draw[thick] (60:3) arc[radius = 3, start angle=60, delta angle=240]
   (30:3) arc[radius=3, start angle = 30, delta angle=-60];
  \node[above left] at (120:3) {$\omega_y$};
  \draw[->] (-4,0) -- (4,0);
  \draw[->] (0,-4) -- (0,4);
  
  \path (2.5,-2.5) node(a){$\mathcal D$};
  \draw[->] (a) -- (0.5,-2);
  \draw[->] (a) -- (2,-0.5);
  \draw[->] (-1.2,-1) node[below]{$D(0,e^{-\pi T})$} -- (-0.2,-0.2);
\end{tikzpicture}
  \end{center}
 \end{minipage}\hfill%
 \begin{minipage}[c]{0.5\textwidth}
  \caption{The domain $\mathcal D$ for the Grushin equation in the rectangle. The 
equivalent of the $U$ of Section \ref{sec:main_proof} is $U=\{0<|z|<1, |\arg(z)|\in 
\omega_y\}$. We still can't control entire functions in $D(0,e^{-\pi T})$ from their 
$L^2$ norm in $U^\delta$ if $\delta$ is smaller than 
$e^{-\pi T}$.}\label{fig:mathcal_D_2}
 \end{minipage}
\end{figure}
Therefore, we can do the same proof as in Section \ref{sec:main_proof}, but with 
$\mathcal D = \{e^{-\pi T}<|z|<1, |\arg(z)|\in \omega_y\}$ (see figure 
\ref{fig:mathcal_D_2}) and $U = \{0<|z|<1, |\arg(z)|\in \omega_y\}$.
\end{proof}

\paragraph{Acknowledgments:} The author thanks his Ph.D. advisor, Gilles Lebeau, for 
the relation between the Grushin equation and the toy model, as well as many other 
ideas, and many fruitful discussions. He also thanks the referees, that provided a 
lot of comments that greatly improved the article.
\printbibliography
\end{appendices}

\end{document}